\documentclass[11pt,onecolumn]{article}

\usepackage{a4wide}
\usepackage{amsmath,amsthm,epsfig,amssymb,amsbsy}
\usepackage{enumerate}
\usepackage{comment}
\usepackage{algorithm}
\usepackage{algorithmic}
\usepackage{amsfonts}       
\usepackage{dsfont}
\usepackage{enumitem}
\usepackage{amssymb}

\usepackage{threeparttable}
\usepackage{chngcntr}
\usepackage{tikz}
\usepackage{tikz-3dplot}
\usepackage{pgfplots}
\usetikzlibrary{backgrounds}
\usepackage{makecell}
\usepackage{tabularx}

\usepackage[giveninits=true,maxbibnames=9, maxcitenames=2, backend=biber]{biblatex}
\newenvironment{keywords}{\begin{paragraph}{Keywords:}
}
{
\end{paragraph}
}
    
\newenvironment{subclass}{\begin{paragraph}{AMS Subject Classification:}
}
{\end{paragraph}
}

\DeclareMathOperator*{\argmin}{arg\,min}

\DeclareMathOperator*{\argmax}{arg\,max}

\DeclareMathOperator{\gph}{gph}

\DeclareMathOperator{\dom}{dom}

\DeclareMathOperator{\intr}{int}
\DeclareMathOperator{\bdry}{bdry}

\DeclareMathOperator{\con}{con}

\DeclareMathOperator{\ran}{rge}

\newcommand{\bR}{\mathbb{R}}

\newcommand{\bN}{\mathbb{N}}

\newcommand{\exR}{\overline{\mathbb{R}}}

\newcommand{\cC}{\mathcal{C}}

\newcommand{\cN}{\mathcal{N}}

\newcommand{\cI}{\mathcal{I}}
\newcommand{\cJ}{\mathcal{J}}

\newcolumntype{Y}{>{\centering\arraybackslash}X}

\newcommand{\phif}{\nabla_\Phi f}

\makeatletter
\newcommand{\prox}[3][\@nil]{%
  \def\tmp{#1}%
   \ifx\tmp\@nnil
       \operatorname{prox}_{#3}^{#2}
    \else
         \operatorname{prox}_{#3}^{#1 \star #2}
    \fi}

\newcommand{\bprox}[3][\@nil]{%
  \def\tmp{#1}%
   \ifx\tmp\@nnil
       \operatorname{bprox}_{#3}^{#2}
    \else
        \operatorname{bprox}_{#3}^{#1 #2}
    \fi}
\makeatother

\usepackage{hyperref}
\hypersetup{
    colorlinks=true,
    linkcolor=blue,
    filecolor=magenta,
    citecolor =magenta,  
    urlcolor=magenta,
    pdftitle={Generalized FBE}
    }

\usepackage[capitalize]{cleveref}[0.19]

\usepackage{aliascnt}

\crefname{section}{section}{sections}
\crefname{subsection}{subsection}{subsections}
\Crefname{section}{Section}{Sections}
\Crefname{subsection}{Subsection}{Subsections}

\Crefname{figure}{Figure}{Figures}

\crefformat{equation}{\textup{#2(#1)#3}}
\crefrangeformat{equation}{\textup{#3(#1)#4--#5(#2)#6}}
\crefmultiformat{equation}{\textup{#2(#1)#3}}{ and \textup{#2(#1)#3}}
{, \textup{#2(#1)#3}}{, and \textup{#2(#1)#3}}
\crefrangemultiformat{equation}{\textup{#3(#1)#4--#5(#2)#6}}%
{ and \textup{#3(#1)#4--#5(#2)#6}}{, \textup{#3(#1)#4--#5(#2)#6}}{, and \textup{#3(#1)#4--#5(#2)#6}}

\Crefformat{equation}{#2Equation~\textup{(#1)}#3}
\Crefrangeformat{equation}{Equations~\textup{#3(#1)#4--#5(#2)#6}}
\Crefmultiformat{equation}{Equations~\textup{#2(#1)#3}}{ and \textup{#2(#1)#3}}
{, \textup{#2(#1)#3}}{, and \textup{#2(#1)#3}}
\Crefrangemultiformat{equation}{Equations~\textup{#3(#1)#4--#5(#2)#6}}%
{ and \textup{#3(#1)#4--#5(#2)#6}}{, \textup{#3(#1)#4--#5(#2)#6}}{, and \textup{#3(#1)#4--#5(#2)#6}}

\newtheorem{theorem}{Theorem}[section]

\newaliascnt{lemma}{theorem}
\newtheorem{lemma}[lemma]{Lemma}
\aliascntresetthe{lemma}

\newaliascnt{proposition}{theorem}
\newtheorem{proposition}[proposition]{Proposition}
\aliascntresetthe{proposition}

\newaliascnt{corollary}{theorem}
\newtheorem{corollary}[corollary]{Corollary}
\aliascntresetthe{corollary}

\newaliascnt{assumption}{theorem}
\newtheorem{assumption}[assumption]{Assumption}
\aliascntresetthe{assumption}

\newaliascnt{definition}{theorem}
\newtheorem{definition}[definition]{Definition}
\aliascntresetthe{definition}

\newaliascnt{fact}{theorem}
\newtheorem{fact}[fact]{Fact}
\aliascntresetthe{fact}

\newtheoremstyle{boldremark}
    {\dimexpr\topsep/2\relax} 
    {\dimexpr\topsep/2\relax} 
    {}          
    {}          
    {\bfseries} 
    {.}         
    {.5em}      
    {}          

\theoremstyle{boldremark}
\newaliascnt{remark}{theorem}
\newtheorem{remark}[remark]{Remark}
\aliascntresetthe{remark}

\theoremstyle{boldremark}
\newaliascnt{example}{theorem}
\newtheorem{example}[example]{Example}
\aliascntresetthe{example}

\crefname{theorem}{Theorem}{Theorems}
\Crefname{theorem}{Theorem}{Theorems}
\crefname{lemma}{Lemma}{Lemmas}
\Crefname{lemma}{Lemma}{Lemmas}
\crefname{proposition}{Proposition}{Propositions}
\Crefname{proposition}{Proposition}{Propositions}
\crefname{corollary}{Corollary}{Corollaries}
\Crefname{corollary}{Corollary}{Corollaries}
\crefname{definition}{Definition}{Definitions}
\Crefname{definition}{Definition}{Definitions}
\crefname{fact}{Fact}{Facts}
\Crefname{fact}{Fact}{Facts}
\crefname{remark}{Remark}{Remarks}
\Crefname{remark}{Remark}{Remarks}
\crefname{example}{Example}{Examples}
\Crefname{example}{Example}{Examples}

\newlist{lemenum}{enumerate}{1} 
\setlist[lemenum]{label=(\roman*), ref=\thelemma(\roman*), font=\rm}
\crefalias{lemenumi}{lemma}

\newlist{propenum}{enumerate}{1} 
\setlist[propenum]{label=(\roman*), ref=\theproposition(\roman*), font=\rm}
\crefalias{propenumi}{proposition}

\newlist{thmenum}{enumerate}{1} 
\setlist[thmenum]{label=(\roman*), ref=\thetheorem(\roman*), font=\rm}
\crefalias{thmenumi}{theorem}
\counterwithin*{thmenumi}{theorem}

\newlist{defenum}{enumerate}{1} 
\setlist[defenum]{label=(\roman*), ref=\thedefinition(\roman*), font=\rm}
\crefalias{defenumi}{definition}

\newlist{corenum}{enumerate}{1} 
\setlist[corenum]{label=(\roman*), ref=\thecorollary(\roman*), font=\rm} 
\crefalias{corenumi}{corollary}

\newlist{examplenum}{enumerate}{1} 
\setlist[examplenum]{label=(\roman*), ref=\theexample(\roman*), font=\rm} 
\crefalias{examplenumi}{example}

\usepackage{xcolor}
\newcommand{\change}[2]{#2}

\newcommand\blfootnote[1]{%
  \begingroup
  \renewcommand\thefootnote{}\footnote{#1}%
  \addtocounter{footnote}{-1}%
  \endgroup
}

\title{On the Regularity of Generalized Conjugate Functions}
\author{Konstantinos Oikonomidis\thanks{KU Leuven,
		Department of Electrical Engineering (ESAT-STADIUS),
		Kasteelpark Arenberg 10, 3001 Leuven, Belgium~
		{\tt%
            \href{mailto:konstantinos.oikonomidis@esat.kuleuven.be}{\{konstantinos.oikonomidis,}%
			\href{mailto:panos.patrinos@esat.kuleuven.be}{panos.patrinos\}}%
			\href{mailto:konstantinos.oikonomidis@esat.kuleuven.be,panos.patrinos@esat.kuleuven.be}{@esat.kuleuven.be}%
		}
	} \and Emanuel Laude\footnotemark[2]\thanks{Proxima Fusion GmbH,
		Fl\"o{\ss}ergasse 2,
		81369 Munich, Germany~
		{\tt
        \href{mailto:elaude@proximafusion.com}{elaude@proximafusion.com}%
		}
	} \and Panagiotis Patrinos\footnotemark[1]}

\begin{document}

\maketitle
\blfootnote{Work supported by: the Research Foundation Flanders (FWO) research projects G081222N, G033822N; Research Council KUL grant C14/24/103.}

\begin{abstract}
We investigate regularity properties of generalized conjugate functions induced by a general coupling function and the associated generalized proximal mapping. Our main results provide verifiable conditions ensuring local single-valuedness, continuity, Lipschitz continuity, and differentiability of the generalized proximal mapping, and transfer these properties to generalized conjugates providing explicit derivative formulas. These results are based on a nonsmooth implicit function theorem for generalized equations, relying on graphical localizations and second-order variational tools. Beyond first-order regularity, we also derive conditions under which generalized conjugates are strictly twice differentiable.
\end{abstract}
\begin{subclass}
65K05 $\cdot$ 49J52 $\cdot$ 90C30
\end{subclass}
\tableofcontents

\section{Introduction}
\subsection{Motivation and related work}
Envelope functions are important tools in optimization and variational analysis due to their regularization and smoothing properties. The celebrated Moreau envelope \cite[Definition 1.22]{RoWe98} is known to, given a convex but possibly discontinuous or even extended real-valued function, return a real-valued and smooth approximation whose stationary points are exactly the minimizers of the original function. In this sense, the envelope provides a smooth surrogate that retains the essential features of the underlying problem.

The situation differs greatly in the nonconvex setting where except for the class of hypoconvex functions, those that are convex up to adding a quadratic of positive curvature, the continuity and smoothness properties of the envelope cannot be guaranteed in general. Classical results overcome this implication by considering so-called prox-regular functions \cite[Definition 13.27]{RoWe98} that in a local sense behave similarly to (hypo)convex functions. The notion of prox-regularity was itself introduced in \cite{poliquin1996prox} and has been shown to hold for large function classes such as strongly amenable functions \cite{poliquin1996prox}, becoming a standard assumption in the related literature.

Nevertheless, regularization techniques extend far beyond the squared Euclidean norm. In fact, general coupling functions of the form $\Phi(x,y)$ were used for envelope functions by Moreau in his 1970 seminal paper \cite{moreau1970inf}. Maybe the most well-known extensions of the Euclidean envelopes are the so-called left and right Bregman--Moreau envelopes that utilize the Bregman divergence $D_h$ generated by a (Legendre) convex kernel function $h$. This type of regularization has been studied extensively, with \cite{bauschke2018regularizing} focusing on the convex setting and \cite{chen2012moreau} on the relatively hypoconvex setting as an extension of standard hypoconvexity. In \cite{laude2020bregman} the notion of prox-regularity was extended to the Bregman setting and local continuous differentiability of the envelope as well as Lipschitz continuity of the proximal mapping were established, while in \cite{ahookhosh2021bregman} the (strict) twice differentiability of the envelope was studied. A similar regularization technique utilizes Csisz\'ar $f$-divergences instead of Bregman ones, with classical works \cite{teboulle1992entropic,iusem1994entropy,iusem1995convergence,teboulle1997convergence} focusing on the convex setting and the nonconvex setting remaining largely unexplored.

Another interesting case is the so-called anisotropic Moreau envelope \cite{combettes2013moreau,laude2019optimization,laude2021lower,laude2023dualities,laude2025anisotropic}, which in many important instances boils down to an infimal convolution (or epi-addition) \cite[Chapter 1]{RoWe98} between the function $f$ and a (Legendre) convex regularizer $\phi$. The regularity properties of such functions have been studied extensively in the convex setting \cite{bauschke2017correction}, where an interesting extension of Moreau's decomposition theorem links the corresponding proximal mapping with the aforementioned Bregman proximal mapping \cite[Proposition 3.9]{laude2025anisotropic}. In the nonconvex setting, local continuous differentiability of the envelope was established in \cite{laude2019optimization,laude2021lower} under suitable extensions of prox-regularity and prox-boundedness. Moreover, in recent works \cite{laude2023dualities,laude2025anisotropic,oikonomidis2025nonlinearly,oikonomidis2025nonlinearlymom} anisotropic envelopes were connected to a generalized smoothness notion that is less restrictive than Lipschitz smoothness and is tailored to popular neural network optimization algorithms. 

The simple fact that the aforementioned properties depend on the choice of the coupling $\Phi$ raises the following question:
\begin{quote}
    \emph{Which envelope regularity properties persist under genuinely abstract couplings, and under what minimal assumptions?}
\end{quote}

In this paper we answer this question through the lens of $\Phi$-convexity. Building on \cite{oikonomidis2025forwardbackward}, which introduced a unifying framework for continuous optimization methods based on $\Phi$-convexity theory, we study the continuity and (twice) continuous differentiability of $\Phi$-conjugates (generalized envelopes) and $\Phi$-subdifferentials (generalized forward steps). Our work not only unifies the study of generalized envelopes and leads to new theoretical results but also has practical implications, as we demonstrate next.

\textbf{Analysis of proximal and subgradient methods.} On the one hand, classical approaches analyze the proximal point algorithm as an instance of gradient descent on the (smooth) Moreau envelope, simplifying and streamlining the analysis, see for example \cite{nedelcu2014computational}. On the other, \cite{davis2019stochastic} introduced a line of works where the Moreau envelope plays the role of the smooth surrogate for (stochastic) subgradient methods beyond convexity, closing a gap in the literature and highlighting its usefulness. Our results on the differentiability of the $\Phi$-conjugates can thus potentially lead to new (proximal) subgradient algorithms and tighter convergence guarantees for existing ones.

\textbf{Forward-backward envelopes and fast schemes.} Envelope functions have successfully been used in designing fast algorithms for additive composite minimization problems. In this context, the forward-backward envelope (FBE) \cite{stella2017forward,stella2017simple,themelis2018forward} plays the role of the continuous, and in many cases even smooth, merit function that is used to globalize some quasi-Newton procedure. Even in the simple case of minimizing the sum of a Lipschitz smooth function and a nonsmooth regularizer, the FBE inherits the regularity properties of the Moreau envelope and the gradient step of the algorithm \cite{themelis2018forward}. Generalizing this approach to the $\Phi$-convexity setup of \cite{oikonomidis2025forwardbackward}, the regularity properties of generalized conjugates and gradient steps can be utilized to create fast quasi-Newton algorithms for very large classes of functions.

\textbf{Optimal transport.} Generalized convexity (known as $c$-concavity in that context) plays an instrumental role in optimal transport (OT) theory. Classical results on Kantorovich duality relate the solution of the OT problem with that of a maximization over $\Phi$-convex functions \cite[Theorem 5.10]{Vil08}. Under suitable assumptions on the given cost ($\Phi$ in our notation) and measures, the joint probability measure is concentrated in the graphs of $\Phi$-subdifferentials \cite[Theorem 5.10]{Vil08}. Moreover, seminal works \cite{gangbo1996geometry,ma2005regularity} study the regularity properties of $\Phi$-convex functions and transformations as they arise naturally from a geometric perspective. It is thus straightforward that the regularity properties of $\Phi$-conjugates and $\Phi$-subdifferentials directly influence those of the corresponding transport maps.

\subsection{Outline and summary of contributions}
\Cref{sec:prelim} introduces concepts and notions from nonsmooth analysis and \Cref{sec:gen_convexity} introduces notions and results from generalized convexity theory. Various examples of coupling functions $\Phi$ along with the $\Phi$-conjugates and $\Phi$-subdifferentials are presented.

In \Cref{sec:phi_subg} we provide sufficient conditions for the existence, continuity and differentiability of $\Phi$-subgradients. More precisely, we study $\varepsilon$-$\Phi$-subgradients as approximate minimization problems using Ekeland's variational principle in \cref{prop:phi_subg_existence}, which is the main result of this section.

In \Cref{sec:diff_phi_conj} we study the continuity and differentiability properties of the $\Phi$-conjugates and the generalized proximal mapping. In \cref{subsec:cont_diff_conj}, we introduce a generalized form of prox-regularity which we then utilize to prove local continuous differentiability of $\Phi$-conjugates under various equivalent conditions in \cref{thm:equivalence_phi_prox_reg}. We then proceed in \cref{subsec:lipschitz_diff_conj} to obtain the local Lipschitz continuity of the generalized proximal mappings as well as the strict twice differentiability of the $\Phi$-conjugates using the nonsmooth version of the implicit function theorem in \cref{thm:single_valued_prox}. To the best of our knowledge these results are new for general coupling functions $\Phi$.

\section{Notation and preliminaries}
\label{sec:prelim}
We denote by $\langle \cdot, \cdot\rangle$ the standard Euclidean inner product on $\bR^n$ and by $\|\cdot\|:=\sqrt{\langle \cdot, \cdot \rangle}$ the standard Euclidean norm on $\bR^n$ as well as the spectral norm for matrices. The extended real line is denoted by $\exR = \bR \cup \{-\infty, +\infty\}$. The open ball of radius $r > 0$ centered at $x \in \bR^n$ is denoted as $B(x, r)$, while the closed ball as $\overline{B}(x, r)$. For a point $x \in \bR^n$ we denote by $\mathcal{N}(x)$ the collection of all neighborhoods of $x$. The \emph{effective domain} of an extended real-valued function $f : \bR^n \to \exR$ is denoted by $\dom f:=\{x\in\bR^n : f(x)<\infty\}$, and we say that $f$ is \emph{proper} if $\dom f\neq\emptyset$ and $f(x) > -\infty$ for all $x \in \bR^n$; \emph{lower semicontinuous (lsc)} if $f(\bar x)\leq\liminf_{x\to\bar x}f(x)$ for all $\bar x\in\bR^n$. Following \cite[Definition 1.33]{RoWe98}, we say that $f:\bR^n \to \exR$ is \emph{locally lsc} at $\bar x$, a point where $f(\bar x)$ is finite, if there is an $\varepsilon > 0$ such that all sets of the form $\{x \in \overline B(\bar x, \varepsilon) \mid f(x) \leq \alpha\}$ with $\alpha \leq f(\bar x) + \varepsilon$ are closed. We define by $\Gamma_0(\bR^n)$ the class of all proper, lsc convex functions $f:\bR^n \to \exR$ and with $\mathcal{C}^k(Y)$ the ones which are $k$ times continuously differentiable on an open set $Y \subseteq \bR^n$. For a proper function $f :\bR^n \to \exR$ and $\lambda \geq 0$ we define the epi-scaling $(\lambda \star f)(x) = \lambda f(\lambda^{-1} x)$ for $\lambda > 0$ and $(\lambda \star f)(x)=\delta_{\{0\}}(x)$ otherwise. We say that a proper, lsc and convex function $f:\bR^n \to \exR$ is \emph{essentially smooth} if $f$ is differentiable on $\intr (\dom f)$ such that $\|\nabla f(x^\nu)\|\to \infty$ whenever $\intr (\dom f) \ni x^\nu \to x \in \bdry (\dom f)$, \emph{essentially strictly convex} if $f$ is strictly convex on every convex subset of $\partial f$, and \emph{Legendre}, if $f$ is both essentially smooth and essentially strictly convex. Following \cite[p.~35]{dontchev2009implicit}, a function $F: \bR^n \to \bR^m$ is said to be strictly differentiable at a point $\bar x$ if there exists a linear mapping $A:\bR^n \to \bR^m$ such that for every $\varepsilon > 0$ there exists a $V \in \mathcal{N}(\bar x)$ with $\|F(x)-F(x')-A(x-x')\| \leq \varepsilon \|x-x'\|$,  for all $x, x' \in V$. Note that if $F \in \cC^1(V)$ for some $V \in \cN(\bar x)$ then it is \emph{strictly differentiable} at $\bar x$ but the reverse implication does not hold in general.
Let $F: \bR^n \rightrightarrows \bR^m$ be a set-valued mapping. We define its \emph{domain} $\dom F := \{x \in \bR^n \mid F(x) \neq \emptyset\}$, its \emph{range} $\ran F := \{u \in \bR^m \mid u \in F(x) \text{ for some } x \in \bR^n\}$ and its \emph{graph} $\gph F := \{(x,u) \in \bR^n \times \bR^m \mid u \in F(x)\}$. The \emph{inverse} mapping $F^{-1}:\bR^m \rightrightarrows \bR^n$ is defined by $F^{-1}(u) = \{x \in \bR^n \mid u \in F(x)\}$. Clearly, $(F^{-1})^{-1} = F$. Following \cite[p. 264]{RoWe98} we denote the set of points that minimize a function $f$ to within $\varepsilon$ by $\varepsilon$-$\argmin f:= \{x \mid f(x) \leq \inf f + \varepsilon\}$.

\section{Generalized convexity}
\label{sec:gen_convexity}
In this section we recapitulate the existing notions of $\Phi$-convexity and $\Phi$-conjugacy \cite{moreau1970inf} which are used heavily as tools in the remainder of the manuscript. Originating as a generalization of convexity to nonlinear spaces, these notions have since appeared in the context of eliminating duality gaps in nonconvex and nonsmooth optimization or optimal transport theory; see, e.g., \cite{rockafellar1974augmented,penot1990strongly,Vil08,bauermeister2021lifting,figalli2011multidimensional}. Recently, $\Phi$-convexity theory has also been extensively utilized in analyzing and designing optimization methods \cite{laude2025anisotropic,leger2023gradient,oikonomidis2025forwardbackward,bednarczuk2025proximal,oikonomidis2025nonlinearly,bednarczuk2026primal}. 

\begin{definition}[$\Phi$-convex and $\Phi$-concave functions] \label{def:phi_cvx}
Let $X$ and $Y$ be nonempty sets and $\Phi: X \times Y \to \bR$ a real-valued coupling. Let $f : X \to \exR$ and $g: Y \to \exR$.
We say that $f$ is $\Phi$-convex on $X$ if there is an index set $\cI$ and parameters $(y_i, \beta_i) \in Y \times \exR$ for $i \in \cI$ such that
\begin{align}
    f(x) = \sup_{i \in \cI} \Phi(x, y_i) - \beta_i\quad \forall x \in X.
\end{align}
If $\cI =\emptyset$ we say that $f \equiv -\infty$ is $\Phi$-convex on $X$ by definition.
We say that $g$ is $\Phi$-convex on $Y$ if there is an index set $\cJ$ and parameters $(x_j, \alpha_j) \in X \times \exR$ for $j \in \cJ$ such that
\begin{align}
g(y) = \sup_{j \in \cJ} \Phi(x_j, y) - \alpha_j \quad \forall y \in Y.
\end{align}
If $\cJ =\emptyset$ we say that $g \equiv -\infty$ is $\Phi$-convex on $Y$ by definition.
We say that $f$ or $g$ is $\Phi$-concave if $-f$ or $-g$ is $\change{\Phi}{(-\Phi)}$-convex.
\end{definition}
If $X=Y$ are indistinguishable and $\Phi$ is not symmetric we shall refer to $f$ as left $\Phi$-convex and $g$ as right $\Phi$-convex. Note that in the bibliography $\Phi$-convex functions appear also as $\Phi$-envelopes. In light of \cite[Theorem 8.13]{RoWe98}, any proper, lsc and convex function $h:\bR^n \rightarrow \exR$ is the pointwise supremum over a family of affine functions:
\begin{equation}
    h(x) = \sup_{i \in \mathcal{I}}\langle x, y_i\rangle - \beta_i.
\end{equation}
Therefore, using the coupling $\Phi = \langle \cdot,\cdot \rangle$ and identifying $X, Y$ with $\bR^n$, one recovers from $\Phi$-convex functions the class of proper, lsc and convex functions and $(-\Phi)$-concavity coincides with the classical notion of concavity. On the other hand, if $\Phi(\cdot, y)$ is lsc and convex for any $y \in \bR^n$, a $\Phi$-convex function is lsc and convex.

Central to our analysis are the notions of $\Phi$-conjugate and $\Phi$-biconjugate functions:
\begin{definition}[$\Phi$-conjugate functions]
Let $X$ and $Y$ be nonempty sets and $\Phi: X \times Y \to \bR$ a real-valued coupling. Let $f: X \to \exR$. Then we define
\begin{align}
f^\Phi(y)=\sup_{x \in X} \Phi(x, y) - f(x),
\end{align}
as the $\Phi$-conjugate of $f$ on $Y$ and 
\begin{align}
f^{\Phi\Phi}(x)=\sup_{y \in Y} \Phi(x, y) - f^\Phi(y),
\end{align}
as the $\Phi$-biconjugate back on $X$.
The definitions of $g^\Phi$ and $g^{\Phi\Phi}$ for $g:Y \to \exR$ are parallel.
\end{definition}
Equivalently to the $\Phi$-convexity definition, if $X=Y$ are indistinguishable and $\Phi$ is not symmetric we shall refer to $f^\Phi$ as the left $\Phi$-conjugate and $g^\Phi$ as the right $\Phi$-conjugate.
From the definition it is clear that $f^\Phi$ is $\Phi$-convex on $Y$ and $f^{\Phi\Phi}$ is $\Phi$-convex back on $X$. Notice that when $X = Y = \bR^n$, by considering the standard inner product $\langle \cdot, \cdot \rangle$ as a coupling, we obtain the classical convex conjugate definition \cite[Equation 11(1)]{RoWe98}. We also adopt from \cite{balder1977extension} the definition of the $\Phi$-subgradient, which is a generalization of the classical notion of subgradient:

\begin{definition}[$\varepsilon$-$\Phi$-subgradients] \label{def:phi_subg}
Let $X$ and $Y$ be nonempty sets and $\Phi: X \times Y \to \bR$ a real-valued coupling. Let $f: X \to \exR$ and $\varepsilon \geq 0$. Then we say that $y$ is a $\varepsilon$-$\Phi$-subgradient of $f$ at $\bar x$ if
\begin{equation} \label{eq:phi_subgradient_ineq}
    f(x) \geq f(\bar x) + \Phi(x, y) - \Phi(\bar x, y) - \varepsilon,
\end{equation}
for all $x \in X$.

We denote by $\partial_\Phi^\varepsilon f(\bar x)$ the set of all $\varepsilon$-$\Phi$-subgradients of $f$ at a point $\bar{x} \in X$.

If $\varepsilon =0$ we omit the superscript $\varepsilon$ and refer to $\partial_\Phi f(\bar x)$ as the $\Phi$-subdifferential of $f$ at $\bar x$ and $y \in \partial_\Phi f(\bar x)$ as a $\Phi$-subgradient of $f$ at $\bar x$.

If $\partial_\Phi f(\bar x) = \{y\}$ is a singleton we adopt the notation $\nabla_\Phi f(\bar x) = y$.
\end{definition}
When $X = Y = \bR^n$ and $\Phi(x,y)=\langle x,y \rangle$, the above definition coincides with the classical notion of $\varepsilon$-subgradients of convex analysis \cite[Definition 5.8]{mordukhovich2022convex}.

The following statement is a generalization of the Fenchel--Young inequality. It is standard in the literature; see \cite{balder1977extension,dolecki1978convexity} and references therein. 
\begin{proposition}[$\Phi$-Fenchel--Young inequality] \label{thm:phi_envelope}
Let $X$ and $Y$ be nonempty sets, $\Phi: X \times Y \to \bR$ a real-valued coupling and $f:X\to \exR$. Then we have
\begin{propenum}
\item \label{thm:phi_envelope:phi_convex} $f^\Phi$ is $\Phi$-convex on $Y$ and $f^{\Phi\Phi}$ is $\Phi$-convex on X;
\item \label{thm:phi_envelope:fenchel_young} $f(x) + f^\Phi(y) \geq \Phi(x, y) \quad \forall x \in X, \quad \forall y \in Y$;
\item \label{thm:phi_envelope:phi_biconj} $f(x) \geq f^{\Phi\Phi}(x) \quad \forall x \in X$.
\end{propenum}
In addition, $f^{\Phi\Phi}$ is the pointwise largest $\Phi$-convex function below $f$. In particular, this means that $f$ is $\Phi$-convex on $X$ if and only if $f(x) = f^{\Phi\Phi}(x)$ for all $x \in X$.

The statements for $g:Y \to \exR$ are parallel.
\end{proposition}

The following proposition relates $\Phi$-subgradients with minimization problems and is of major importance for our analysis. It is standard in the related literature; see \cite[Proposition 3.5]{dolecki1978convexity}.
\begin{proposition}[$\Phi$-subgradient equivalences] \label{thm:phi_subgradients}
    Let $X$ and $Y$ be nonempty sets, $\Phi: X \times Y \to \bR$ a real-valued coupling and  $f: X \to \exR$. Then for any $\bar x \in X$ and $\bar y \in Y$ the following statements are equivalent:
    \begin{propenum}
\item $\bar y \in \partial_\Phi f(\bar x)$; \label{thm:phi_subgradients:y_subg}
\item $ f(\bar x) + f^\Phi(\bar y) = \Phi(\bar x, \bar y)$;
\item $\bar x \in \argmin_{x \in X} f(x) - \Phi(x, \bar y)$; \label{thm:phi_subgradients:x_argmin}
\end{propenum}
where any of the above equivalent statements implies that $f^{\Phi\Phi}(\bar x)= f(\bar x)$ and $\bar x \in \partial_\Phi f^\Phi(\bar y)$.
In particular, if $f$ is $\Phi$-convex $\bar y \in \partial_\Phi f(\bar x) \Leftrightarrow \bar x \in \partial_\Phi f^\Phi(\bar y)$.
\end{proposition}

Note that due to the equivalences shown above, another way of writing the $\Phi$-subdifferential of $f$ at a point $x \in X$ is via $\partial_\Phi f(x) = \{y \in Y \mid f(x)+f^\Phi(y) = \Phi(x, y)\}$, while more generally $\gph \partial_\Phi f = \{(x,y)\in X \times Y \mid f(x)+f^\Phi(y) = \Phi(x, y)\}$. The aforementioned relations indicate the connection of the $\Phi$-subdifferential with the (generalized) proximal mapping via $(\partial_\Phi f)^{-1}(y) = \argmin_{x \in X}f(x)-\Phi(x,y)$.

In order to better describe the concepts related to $\Phi$-convexity we provide specific examples of coupling functions in the following. The first example is from \cite{oikonomidis2025forwardbackward}.
\begin{example}[$\Phi$-convexity with the quadratic coupling] \label{example:quadratic}
    A coupling function that is standard in the related literature is $\Phi(x,y) = -\tfrac{1}{2\gamma}\|x-y\|^2$ for $\gamma > 0$. For this coupling, in light of \cite[Proposition 3.4]{laude2023dualities}, any $\rho$-hypoconvex function is a $\Phi$-convex function for $\gamma \in (0, 1/\rho]$, while its $\Phi$-subdifferential is given by $\partial_\Phi f(x) = x + \gamma \partial f(x)$ for any $x \in \dom f$. The $\Phi$-conjugate of a function $f:\bR^n \to \exR$ is given in this setting by $f^\Phi(y) = \sup_{x \in \bR^n}\Phi(x,y) - f(x) = -\inf_{x \in \bR^n}f(x) + \tfrac{1}{2\gamma}\|x-y\|^2$, which is the negative Moreau envelope of the function $f$ \cite[Definition 1.22]{RoWe98}. Note if $f$ is $\Phi$-convex, we have from \cref{thm:phi_envelope} that $f=f^{\Phi\Phi}$ and thus $f(x) = -\inf_{y \in \bR^n}f^\Phi(y) + \tfrac{1}{2\gamma}\|x-y\|^2$ is a negative Moreau envelope itself.
\end{example}
\begin{example}[$\Phi$-convexity with the left Bregman coupling] \label{example:left_bregman}
    Generalizing the previous example to Bregman divergences, let $h:\bR^n \to \exR$ be Legendre and continuous on $\dom h$ which is closed. Then, taking $X = \dom h$, $Y = \intr \dom h$ and $\Phi(x,y) = -\tfrac{1}{\gamma}D_h(x,y)$ we have $f^\Phi(y) = \sup_{x \in \dom h} -\tfrac{1}{\gamma}D_h(x,y) - f(x) = -\inf_{x \in \bR^n} f(x) + \tfrac{1}{\gamma} D_h(x,y)$, which is the negative left Bregman--Moreau envelope \cite[Definition 3.1]{laude2020bregman}. When $f$ is proper, $\dom h = \bR^n$ and $h$ is supercoercive, it was shown in \cite[Proposition 3.4]{laude2023dualities} that for $\gamma = 1$, $\Phi$-convexity of $f$ is equivalent to $f+h$ being convex. In this case, the $\Phi$-subdifferential is given by $\partial_\Phi f = \nabla h^* \circ(\nabla h + \partial f)$. These somewhat strict assumptions on $h$ have recently been lifted in \cite{themelis2025natural,wang2025bregman} that also consider a $\Phi$-convexity approach.
\end{example}
\begin{example}[$\Phi$-convexity with the right Bregman coupling] 
\label{example:right_bregman}
    Switching the arguments of the Bregman divergence from the previous example we get $\Phi(x,y) = -\tfrac{1}{\gamma} D_h(y,x)$ and the $\Phi$-conjugate becomes $f^\Phi(y) = - \inf_{x \in \bR^n} f(x) + \tfrac{1}{\gamma}D_h(y,x)$, i.e., the negative right Bregman--Moreau envelope from \cite[Definition 3.7]{laude2020bregman}.
\end{example}
\begin{example}[$\Phi$-convexity with the anisotropic coupling] \label{example:aniso}
    Let $\phi:\bR^n \to \bR$ be Legendre and $\dom \phi = \bR^n$. Choosing $\Phi(x,y) = -\gamma \star \phi(x-y)$ leads to $f^\Phi(y) = -\inf_{x \in \bR^n}f(x) + \gamma \star \phi(x-y)$ which is the left anisotropic Moreau envelope as defined in \cite[Definition 3.7]{laude2025anisotropic}. Assuming moreover that $\phi$ is supercoercive, if $f$ is anisotropically weakly convex \cite[Definition 3.8]{laude2023dualities}, then $\partial_\Phi f(x) = x-\nabla \phi^*(-\partial f(x))$.
\end{example}
\begin{example}[$\Phi$-convexity with $\varphi$-divergences] \label{example:phi_div}
    Except for the aforementioned examples, an interesting coupling function that is less explored in the related literature is the $\varphi$-divergence from \cite{teboulle1992entropic,iusem1994entropy}. Consider $\varphi:\bR_{++} \to \bR$ strictly convex and $\cC^1(\bR_{++})$ and such that $\varphi(1) = \varphi'(1) = 0$, $\varphi''(1) > 0$, $\lim_{t\to0}\varphi'(t) = -\infty$ and $\lim_{t \to \infty}\varphi(t)=+\infty$. Clearly, $\varphi(t) > 0$ for $t \neq 1$ and $\varphi$ is an essentially smooth function meaning that $\varphi$ is Legendre. Therefore, $\varphi^*$ is also Legendre and $\varphi'$ is one-to-one from $\bR_{++}$ to $\intr \dom \varphi^* = \bR$. We extend the domain of $\varphi$ to $[0,+\infty)$ continuously by redefining $\varphi$ if needed and define the $\varphi$-divergence generated by $\varphi$, $d_\varphi: \bR^n_{+} \times \bR^n_{++} \to \bR$ as
    \begin{equation}
        d_\varphi(x,y) = \sum_{i=1}^n y_i \varphi(\tfrac{x_i}{y_i}).
    \end{equation}
    Clearly for $x \in \bR^n_{++}$, $d_\varphi \geq 0$ and $d_\varphi(x,y) = 0$ if and only if $x = y$ making it a suitable distance-like function. Choosing now $\Phi(x,y) = -\gamma d_\varphi(x,y)$ and $X=\bR^n_{+},Y=\bR^n_{++}$ we have that $f^\Phi(y) = -\inf_{x \in \bR^n_{++}}\gamma d_\varphi(x,y) + f(x)$. This type of envelope functions was extensively studied for convex $f$ in \cite{teboulle1992entropic,iusem1994entropy,iusem1995convergence,teboulle1997convergence}.
\end{example}
The coupling functions we have seen so far couple (a subset of) $\bR^n$ with (a subset of) $\bR^n$, i.e. $x$ and $y$ are vectors of the same space. Nevertheless, the framework of $\Phi$-convexity allows for more robust parametrizations, where $Y$ is in a higher dimensional space than $X$. A prominent example of such a coupling leads to the basic quadratic transform from \cite[Example 11.66]{RoWe98} that we describe next:
\begin{example}[Basic quadratic transform] \label{example:quadratic_transform}
    Let $X = \bR^n$ and $Y = \bR^n \times \bR$ and $\Phi(x,(v,r)) = \langle x,v \rangle - \tfrac{r}{2}\|x\|^2$. By standard convex conjugacy we then have $f^\Phi(v,r) = (f+\tfrac{r}{2}\|\cdot\|^2)^*(v)$. Therefore, if $f$ is prox-bounded and not identically infinity, $f^\Phi$ is a proper, lsc and convex function on $\bR^n \times \bR$ and thus the transformation couples nonconvex prox-bounded functions with convex ones.
\end{example}
We make the following assumption on the coupling function $\Phi$ and the sets $X,Y$ that we consider valid throughout the text.
\begin{assumption}
    $X \subseteq \bR^n$ is closed with nonempty interior and $Y \subseteq \bR^m$ is nonempty and open. Moreover, $\Phi: X \times Y \to \bR$ is continuous relative to $X \times Y$ and continuously differentiable on $\intr (X) \times Y$.
\end{assumption}
\begin{remark}
    Some properties of $\Phi$-convex functions $f$ follow immediately from the assumptions on $\Phi$. For example, since $\Phi$ is continuous (and more generally lsc in the first argument), any $\Phi$-convex function is lsc as the pointwise supremum of lsc functions \cite[Proposition 1.26]{RoWe98}. Moreover if $\Phi(\cdot,y)$ is convex for all $y \in Y$, $f$ is also convex due to the same reasoning. Another interesting case is when $\|\nabla_{xx}^2 \Phi(x,y)\|$ is bounded for all $x,y$ which implies that $\Phi(\cdot,y)$ is hypoconvex for all $y \in Y$ and thus that $f$ itself is also hypoconvex. In this setting, $f$ is then twice differentiable a.e.\ due to Alexandrov theorem, a case that is standard in the OT literature (see \cite[Example 2]{staudt2025uniqueness}).
\end{remark}

\section{Existence and regularity of \texorpdfstring{$\Phi$}{Phi}-subgradients}
\label{sec:phi_subg}
In this section we study the existence of $\Phi$-subgradients along with their continuity and differentiability properties. Under suitable assumptions, the existence and form of $\Phi$-subgradients for the Bregman and anisotropic coupling functions (\cref{example:left_bregman} and \cref{example:aniso}) was studied in \cite{laude2023dualities}. More examples of coupling functions can be found in \cite{oikonomidis2025forwardbackward}. We generalize these results even further based on the notion of a \emph{strong twist} adapted from optimal transport \cite[p. 313]{Vil08}.

\begin{definition}[Twist] \label{def:twist}
Let $\Phi :X \times Y \to \bR$, $X, Y \subseteq \bR^n$ be continuously differentiable on $\intr (X) \times Y$. Then we say that
\begin{defenum}
    \item \label{def:twist:twist} $\Phi$ is twist if for any $x \in \intr X$, $\nabla_x \Phi(x,\cdot) : Y \to \bR^n$ is one-to-one.
    \item \label{def:twist:strong_local} $\Phi$ is local strong twist at $(\bar x, \bar y) \in \intr (X) \times Y$ if for $\bar v :=\nabla_x \Phi(\bar x, \bar y)$ 
    there exists a unique continuously differentiable mapping $G(x,v)$ such that $G(\bar x, \bar v) = \bar y$ with $\nabla_x \Phi(x, G(x,v)) = v$ in a neigborhood of $(\bar x,\bar v)$.
    \item \label{def:twist:strong_global} $\Phi$ is (global) strong twist if it is twist and locally strong twist at any $(\bar x, \bar y) \in \intr (X) \times Y$. In particular this implies that $(\nabla_x \Phi(x, \cdot))^{-1}(v) := \{y \in Y \mid \nabla_x \Phi(x,y) - v=0\}$ is single-valued and continuously differentiable in $(x,v)$.
\end{defenum} 
\end{definition}

\begin{example}[Twist examples] \label{example:twist}
Most of the coupling functions presented in \Cref{sec:gen_convexity} are actually global strong twists under suitable regularity assumptions, as we show next:
    \begin{examplenum}
        \item In the setting of \cref{example:quadratic}, it is straightforward that $\nabla_x \Phi(x,y) = -\tfrac{1}{\gamma}(x-y)$ and thus $(\nabla_x \Phi(x, \cdot))^{-1}(v) = x + \gamma v$.
        \item In the setting of \cref{example:left_bregman}, $\nabla_x \Phi(x,y) = -\tfrac{1}{\gamma}(\nabla h(x)-\nabla h(y))$ and thus $(\nabla_x \Phi(x, \cdot))^{-1}(v) = \nabla h^*(\nabla h(x) + \gamma v)$.
         \item In the setting of \cref{example:right_bregman}, assume moreover that $\dom h = \bR^n$ and $\nabla^2 h$ is positive-definite in $\bR^n$. Then,  $\nabla_x \Phi(x,y) = \tfrac{1}{\gamma}(\nabla^2 h(x)y-\nabla^2 h(x)x)$ and thus $(\nabla_x \Phi(x, \cdot))^{-1}(v) = x + \gamma [\nabla^2 h(x)]^{-1} v$.
         \item In the setting of \cref{example:aniso}, $\nabla_x \Phi(x,y) = -\nabla \phi(\tfrac{1}{\gamma}(x-y))$ and thus $(\nabla_x \Phi(x, \cdot))^{-1}(v) = x - \gamma \nabla \phi^*(-v)$.
         \item In the setting of \cref{example:phi_div}, $[\nabla_x \Phi(x,y)]_i = -\varphi'(x_i/y_i)$ and thus $[(\nabla_x \Phi(x, \cdot))^{-1}(v)]_i= x_i/{\varphi^*}'(-v_i)$.
    \end{examplenum}
\end{example}

In light of the implicit function theorem, a sufficient condition for $\Phi$ to be a local strong twist is the nonsingularity of $\nabla^2_{xy}\Phi(x,y)$. This is described in the following lemma:
\begin{lemma}
    Let $\Phi$ be twice continuously differentiable on $\intr (X) \times Y$. If $\nabla^2_{xy}\Phi(\bar x,\bar y)$ is nonsingular, then $\Phi$ is a local strong twist at $(\bar x, \bar y)$.
\end{lemma}
\begin{proof}
    The proof follows directly by the classical implicit function theorem: consider the function $H: \bR^n \times \bR^n \times \bR^n \to \bR^n$ given by $H(x, v, y) = \nabla_x \Phi(x,y) - v$. Then, $H(\bar x, \bar v, \bar y) = 0$ and $\nabla_y H(\bar x, \bar v, \bar y)$ is nonsingular, which implies the claimed result.
\end{proof}
Note that in general, even if for $\bar x \in \intr X$, $\nabla_{xy}^2 \Phi(\bar x,y)$ is nonsingular for any $y \in Y$ the mapping $\nabla_x \Phi(\bar x, \cdot)$ might not be injective and as such $\Phi$ might not be a twist. A sufficient condition for $\nabla_x \Phi(\bar x, \cdot)$ to be injective, assuming that $Y$ is a convex set, is that $\nabla_{xy}^2 \Phi(\bar x,y)$ is positive-definite for each $y \in Y$ as a result of the mean value theorem. This is actually the case in most of the examples we have seen, e.g. for $\Phi(x,y) = -\tfrac{1}{2\gamma}\|x-y\|^2$, clearly $\nabla_{xy}^2 \Phi(x,y) = \tfrac{1}{\gamma}I \succ 0$.

In contrast to the setting of classical convexity, a full-domain (even smooth) $\Phi$-convex function can have an everywhere empty $\Phi$-subdifferential. A standard example of this can be found in \cite{dolecki1978convexity}, where the function $f \equiv \rho$ for some $\rho \in \bR$ and the coupling $\Phi(x, y) = \exp(x-y)$ parametrized by $ y \in \bR$ are considered. It is straightforward that $f$ is $\Phi$-convex, but clearly for any $y \in \bR$ we cannot have $\exp(-y)(\exp(x)-\exp(\bar x)) \leq 0$ for all $x \in \bR$, i.e. $\partial_\Phi f(\bar x) = \emptyset$ for all $\bar x \in \bR$. Another interesting example of a $\Phi$-convex function with an empty subdifferential at a point is $f(x) = -|x|$ where $\Phi$ is as in \cref{example:quadratic_transform}. This function is $\Phi$-convex but $\partial_\Phi f(0) = \emptyset$, as there is no quadratic function below $f$ which is tangent to $f$ at $0$. An interesting case where $\Phi$-convexity implies nonempty $\Phi$-subdifferentials at every point is that of the coupling $\Phi(x,y) = -L \|x-y\|$ with $X=Y=\bR^n$ (or more generaly in metric spaces with $\Phi(x,y) = -L d(x,y)$). In this case every $L$-Lipschitz continuous function is $\Phi$-convex and has a nonempty $\Phi$-subdifferential in light of \cite[Proposition 2.1.6]{pallaschke2013foundations}.

Nevertheless, Ekeland's variational principle can be utilized in order to prove the existence of $\Phi$-subgradients under some extra regularity assumptions on $f$ and for $\Phi$ being a local strong twist. This is described in the following proposition:

\begin{proposition}[existence of $\Phi$-subgradients]
\label{prop:phi_subg_existence}
    Let $\bar x \in \intr X$ and $f: X \to \exR$ be finite, strictly continuous at $\bar x$ and $\Phi$-convex on $X$ w.r.t. $\Phi$ which is twist and a local strong twist at $(\bar x, \bar y)$ for  $\nabla_x \Phi(\bar x, \bar y) \in \partial f(\bar x)$. Then, $\partial_\Phi f(\bar x)$ is nonempty and bounded. 
\end{proposition}
\begin{proof}
    We pattern the proof of \cite[Proposition 3.13]{laude2023dualities}. Since $f$ is $\Phi$-convex, $f = f^{\Phi \Phi}$ and as such $f(\bar x) = \sup_{y \in Y}\Phi(\bar x,y) - f^\Phi(y)$, which further implies that for any $\varepsilon > 0$ there exists a $\bar y_\varepsilon \in Y$ such that
    \begin{align*}
        f(\bar x) - \varepsilon \leq \Phi(\bar x, \bar y_\varepsilon) - f^\Phi(\bar y_\varepsilon) \leq f(\bar x).
    \end{align*}
    Moreover, in light of the $\Phi$-Fenchel--Young inequality, we have that $f(x) + f^\Phi(\bar y_\varepsilon) \geq \Phi(x , \bar y_\varepsilon)$ for any $x \in X$, which means that $- f^\Phi(\bar y_\varepsilon) \leq f(x) - \Phi(x , \bar y_\varepsilon)$. Substituting this in the inequality above, we obtain:
    \begin{equation} \label{eq:lim_arg}
        f(\bar x) - \Phi(\bar x, \bar y_\varepsilon) \leq f(x) - \Phi(x , \bar y_\varepsilon) + \varepsilon,
    \end{equation}
    which is exactly the $\varepsilon$-$\Phi$-subgradient inequality \eqref{eq:phi_subgradient_ineq}. This inequality implies that $\bar x \in \varepsilon-\argmin_{x \in X}f(x) - \Phi(x , \bar y_\varepsilon)$ and as such we can invoke Ekeland's variational principle \cite[Theorem 1.1]{ekeland1974variational} with $\lambda := \sqrt{\varepsilon}$ and obtain the existence of a point $\bar x_\varepsilon \in X \cap \overline{B}(\bar x; \sqrt{\varepsilon})$ with $f(\bar x_\varepsilon) - \Phi(\bar x_\varepsilon , \bar y_\varepsilon) \leq f(\bar x) - \Phi(\bar x , \bar y_\varepsilon)$ and $\bar x_\varepsilon := \argmin_X f(\cdot) - \Phi(\cdot , \bar y_\varepsilon) + \sqrt{\varepsilon}\|\cdot - \bar x_\varepsilon\|$. Now, since $\bar x \in \intr X$, by taking $\varepsilon$ small enough such that $\overline{B}(\bar x; \sqrt{\varepsilon}) \subset \intr X$, the optimality conditions of this minimization problem become $0 \in \partial f(\bar x_\varepsilon) - \nabla_x \Phi(\bar x_\varepsilon, \bar y_\varepsilon) + \overline{B}(0; \sqrt{\varepsilon})$, which implies that:
    \begin{equation}
        \nabla_x \Phi(\bar x_\varepsilon, \bar y_\varepsilon) - u_\varepsilon =: \bar v_\varepsilon \in \partial f(\bar x_\varepsilon),
    \end{equation}
    holds for some $u_\varepsilon$ with $\|u_\varepsilon\| \leq \sqrt{\varepsilon}$. Since $f$ is a pointwise supremum over continuous functions it is lsc and as it is also strictly continuous at $\bar x$, $x \mapsto \partial f(x)$ is locally bounded at $\bar x$ \cite[Theorem 9.13]{RoWe98}. Up to extracting a subnet, this implies that $\bar v_\varepsilon \to \bar v$ for some $\bar v \in \bR^n$ as $\varepsilon$ goes to $0$. 

    In light of \cref{def:twist:strong_local}, there exists a unique mapping $G: \bR^n \times \bR^n \to \bR^n$ such that $\nabla_x \Phi(x, G(x, \upsilon)) = \upsilon$ in a neighborhood $U$ of $(\bar x, \bar v)$, which is continuously differentiable on $U$. Now for $u$ close enough to $0$ and $v$ close enough to $\bar v$, $\upsilon = u+v \in U$. Therefore, up to extracting a subnet if needed and since $\Phi$ is a twist, we have that $\bar y_\varepsilon = G(\bar x_\varepsilon,\bar \upsilon_\varepsilon)$. Since now $(\bar x_\varepsilon,\bar \upsilon_\varepsilon) \to (\bar x, \bar v)$ and $G$ is continuous on $U$ we have that $\bar y_\varepsilon \to \bar y$. Finally we can pass to the limit in \eqref{eq:lim_arg} and obtain
    \begin{equation}
        f(\bar x) - \Phi(\bar x, \bar y) \leq f(x) - \Phi(x , \bar y) \qquad \forall x\in \bR^n,
    \end{equation}
    which is the claimed result.

    Now note that from the $\Phi$-subgradient inequality and the fact that $\Phi$ is a twist we have that $\partial_\Phi f(\bar x) \subseteq G(\bar x, \partial f(\bar x))$, while since $f$ is strictly continuous, $\partial f(\bar x)$ is bounded and thus so is $G(\bar x, \partial f(\bar x))$ owing to the continuity of $G$. Therefore, $\partial_\Phi f(\bar x)$ is compact.
\end{proof}
Similar results on the nonemptyness of the $\Phi$-subdifferential under $\Phi$-convexity can be found in the OT theory literature \cite[Proposition C.4]{gangbo1996geometry} and \cite[Theorem 10.24]{Vil08}. From the analysis above it is also clear that $\partial_\Phi f$ is outer semicontinuous at every $\bar x \in \dom \partial_\Phi f$ relative to $\dom \partial_\Phi f$. To better see that take any $\{x^\nu\}_{\nu \in \bN}$ with $x^\nu \in \dom \partial_\Phi f$ such that $x^\nu \to \bar x$ and a corresponding $y^\nu \to \bar y$. Then, from the $\Phi$-subgradient inequality we have
\begin{equation*}
    f(x) - \Phi(x,y^\nu) \geq f(x^\nu) - \Phi(x^\nu, y^\nu) \qquad \forall x \in X.
\end{equation*}
Using now the continuity of $\Phi$ and the semicontinuity of $f$ we obtain that $\bar y \in \partial_\Phi f(\bar x)$ and it is also clear that $\gph \partial_\Phi f$ is closed in $X \times Y$.

Next we refine the results of \cref{prop:phi_subg_existence} in order to obtain a sufficient condition for the continuity and differentiability of the $\Phi$-subdifferential, in the case where it is a singleton.
\begin{corollary}[continuity and differentiability of $\Phi$-subgradients] \label{thm:phi_subg}
    Let $f: X \to \bR$ be $\Phi$-convex on $X$ and assume that $\Phi$ is a global strong twist for $\nabla f(\bar x)$ for every $\bar x \in \intr X$. Then,
    \begin{corenum}
        \item if $f \in \mathcal{C}^1(\intr X)$, $\phif$ is continuous on $\intr X$, with $\phif(x) = G(x, \nabla f(x))$; \label{thm:phi_subg:cont}
        \item if $f \in \mathcal{C}^2(\intr X)$, $\phif$ is continuously differentiable on $\intr X$. \label{thm:phi_subg:diff}
    \end{corenum}
\end{corollary}
\begin{proof}
``\labelcref{thm:phi_subg:cont}'':
In light of \cref{prop:phi_subg_existence}, $\partial_\Phi f(\bar x)$ is nonempty.
    From the definition of $\Phi$-subgradients, it is straightforward that for $\bar y$ to be a $\Phi$-subgradient of $f$ at $\bar x \in \intr X$, it is necessary that $\nabla_x \Phi(\bar x,\bar y) = \nabla f(\bar x)$ and hence $\bar y \in (\nabla_x \Phi(\bar x, \cdot))^{-1}(\nabla f(\bar x))$. Since $(\nabla_x \Phi(x, \cdot))^{-1}(v)$ is single-valued and continuous in $(x, v)$ and $\nabla f$ is continuous, $\phif$ is a continuous function on $\intr X$.
    
``\labelcref{thm:phi_subg:diff}'': It follows from the proof of \cref{thm:phi_subg:cont} along with the fact that $\nabla f$ is continuously differentiable.
\end{proof}

\section{Continuity and differentiability properties of \texorpdfstring{$\Phi$}{Phi}-conjugates}
\label{sec:diff_phi_conj}
In this section we study the regularity properties of $\Phi$-conjugates. We will assume that $g(x) - \Phi(x,y)$ is uniformly level-bounded, i.e. it satisfies the following definition that we adapt from \cite[Definition 1.16]{RoWe98}.
\begin{definition}[Uniform level boundedness]
    A function $f: X \times Y \to \exR$ with values $f(x,y)$ is level-bounded in $x$ locally uniformly in $y$ if for each $\bar y \in Y$ and $\alpha \in \bR$ there is a neighborhood $V \in \mathcal{N}(\bar y)$ along with a bounded set $B \subset X$ such that $\{x \mid f(x, y) \leq \alpha \} \subset B$ for all $y \in V$; or equivalently, there is a neighborhood $V \in \mathcal{N}(\bar y)$ such that the set $\{(x,y) \mid y \in V, f(x,y) \leq \alpha\}$ is bounded in $\bR^n \times \bR^m$.
\end{definition}

We can now move on to formulating the assumptions on $g$ that we consider valid throughout the rest of the paper.
\begin{assumption}\label{assum:g}
    $g : X \to \exR$ is proper, lsc and such that $\xi(x, y) := g(x) - \Phi(x,y)$ is level bounded in $x$ locally uniformly in $y$. For $(\partial_\Phi g)^{-1}(y) := \argmin_{x \in X}g(x) - \Phi(x, y)$ we have that $\ran (\partial_\Phi g)^{-1} \subseteq \intr X$.
\end{assumption}
The uniform level-boundedness assumption amounts to the continuity of the $\Phi$-conjugate as we show in the following subsection and can be considered as an adaptation of standard prox-boundedness to the abstract setting that we study. Indeed, in the case of \cref{example:quadratic} by potentially choosing a smaller $\gamma$ we can guarantee that if $g$ is prox-bounded, $g(x) - \Phi(x,y)$ is level bounded in $x$ locally uniformly in $y$ through \cite[Theorem 1.25]{RoWe98}. Similar reasoning can be applied to the case of \cref{example:left_bregman}.

\subsection{Continuous differentiability of \texorpdfstring{$\Phi$}{Phi}-conjugates} \label{subsec:cont_diff_conj}

The continuity of the $\Phi$-conjugates is a direct consequence of \cref{assum:g}:
\begin{lemma} \label{thm:conj_cont}
    Let $g$ be as in \cref{assum:g}. Then $g^\Phi$ and $(\partial_\Phi g)^{-1}$ have the following properties:
    \begin{lemenum}
        \item \label{thm:conj_cont:dom} $(\partial_\Phi g)^{-1}(y) \neq \emptyset$ is compact for all $y \in \dom g^\Phi = Y$;
        \item \label{thm:conj_cont:conj_cont} The $\Phi$-conjugate $g^\Phi$ is continuous on $Y$;
        \item \label{thm:conj_cont:osc} For any sequence $y^\nu \to y^\star \in Y$ contained in $Y$ and $x^\nu  \in (\partial_\Phi g)^{-1}(y^\nu)$, $\{x^\nu\}_{\nu \in \bN}$ is bounded and all its cluster points lie in $(\partial_\Phi g)^{-1}(y^\star)$.
    \end{lemenum}
\end{lemma}
\begin{proof}
    ``\labelcref{thm:conj_cont:dom}'':
    To begin with, note that $g - \Phi(\cdot,y)$ is level-bounded for every $y \in Y$. Since, moreover, $g$ is lsc and proper and $\Phi(\cdot, y)$ is continuous, it is also lsc and proper. Then, from \cite[Theorem 1.9]{RoWe98}, $\inf g - \Phi(\cdot,y)$ is finite and thus so is $g^\Phi(y) = \sup_{x \in X}\Phi(x,y)-g(x) = -\inf_{x \in X} g(x)-\Phi(x,y)$, adapting the arguments of the proof to the subspace topology, noting that $X$ is closed. The fact that $(\partial_\Phi g)^{-1}$ is nonempty and compact follows from the same theorem.
    
    ``\labelcref{thm:conj_cont:conj_cont}'': For any $y \in Y$, there exists by \cref{assum:g} an $x \in (\partial_\Phi g)^{-1}(y)$ and $\Phi$ is continuous on $X \times Y$, which implies through \cite[Theorem 1.17(c)]{RoWe98} that $\inf_{x \in X}g(x) - \Phi(x,y) = -g^\Phi(y)$ is continuous on $Y$, once again by adapting the arguments to the subspace topology on $X \times Y$. 
    
    ``\labelcref{thm:conj_cont:osc}'': Follows by \cite[Theorem 1.17(b)]{RoWe98}.
\end{proof}
Note that we consider functions $g:X \to \exR$. In what follows, since we are interested in local results for points in $\intr X$, when we use results from \cite{RoWe98} or \cite{dontchev2009implicit} that refer to functions defined on $\bR^n$ we implicitly consider the extension $\tilde g:\bR^n \to \exR$ with $\tilde g(x) =  g(x)$ on $X$ and $+\infty$ on $\bR^n \setminus X$. Similarly for operators $T:X \rightrightarrows \bR^m$ we consider $T(x) = \emptyset$ for $x \notin X$, following the discussion in \cite[p. 149]{RoWe98}.

Moving on to the $\mathcal{C}^1$-properties of the generalized conjugates, we first provide an extension of the notion of prox-regularity, which in the standard setup \cite{poliquin1996generalized} is a fundamental property related to the (local) differentiability of the Moreau envelope in the absence of classical convexity. Nonlinear generalizations of prox-regularity were already introduced in the Bregman \cite{laude2020bregman,laude2021lower} and the anisotropic case \cite{laude2021lower}. They lead to local characterizations of single-valuedness and the existence of a resolvent-type expression of the corresponding nonlinear proximal mapping. 

In the following we will further generalize prox-regularity in a way similar to how the $\Phi$-subgradient inequality generalizes the Euclidean descent lemma as shown in \cite[Example 4.2]{oikonomidis2025forwardbackward}. 

\begin{definition}[$\Phi$-prox-regularity]
    Let $\bar x \in \intr X$ and $\bar y \in Y$.
    A function $g:X \to \exR$ is said to be $\Phi$-prox-regular at $\bar x$ for $(\bar v, \bar y)$ if $g$ is finite and locally lsc at $\bar x$ with $\nabla_x \Phi(\bar x, \bar y) =:\bar v\in\partial g(\bar x)$, and there exists $\varepsilon >0$ such that:
\begin{equation} \label{eq:phi_prox_regularity}
    g(x') \geq g(x) +\Phi(x', y) -\Phi(x, y),
\end{equation}
whenever $\|x' - \bar x\| \leq \varepsilon$, $\|x - \bar x\| \leq \varepsilon$, $g(x) - g(\bar x) \leq \varepsilon$, $v \in \partial g(x)$ and $\|v - \bar v\| \leq \varepsilon$, $\|y - \bar y\| \leq \varepsilon$ and $v=\nabla_x \Phi(x, y)$.
\end{definition}

In many important cases nearness of $y$ to $\bar y$ is entailed by nearness of $(x,v)$ to $(\bar x,\bar v)$ where $y$ is the unique point for which $v = \nabla_x \Phi(x, y)$ and $y$ varies continuously with $(x,v)$ and as such $y$ need not be quantified. This is captured in terms of the notion of strong twist \cref{def:twist} as we highlight in the following.

\begin{remark}[$\Phi$-prox-regularity under strong twist]
Under \cref{def:twist} the $\Phi$-prox-regularity subgradient inequality reduces to
$$
g(x') \geq g(x) +\Phi(x', G(x,v)) -\Phi(x, G(x,v)),
$$
where $y=G(x,v)$ is given in \cref{def:twist:strong_local}. In fact, in this case $G(x,v)=y$ sets up a one-to-one correspondence between $\Phi$-subgradients $y$ and classical subgradients $v$. 
This holds even globally in the anisotropic and the Bregman case: similarly to \cref{example:twist}, for $\Phi(x,y) = -\phi(y-x)$ we have that $G(x,v) = (\nabla_x \Phi(x, \cdot))^{-1}(v) = x + \nabla \phi^*(v)$. After substitution in \cref{eq:phi_prox_regularity} this yields \cite[Definition 2.13]{laude2021lower}. Using the left Bregman divergence as a coupling, we have from \cref{example:twist} that $G(x,v)=(\nabla_x \Phi(x, \cdot))^{-1}(v) = \nabla h^*(\nabla h(x)+\gamma v)$. This yields the notion of relative prox-regularity \cite[Definition 3.16]{laude2020bregman}.
\end{remark}
In the following we show that under suitable assumptions, $\Phi$-prox-regularity is implied by classical prox-regularity generalizing \cite[Proposition 3.17]{laude2020bregman} and \cite[Proposition 2.38]{laude2021lower}.
\begin{proposition}[$\Phi$-prox-regularity from prox-regularity]
    Let $(\bar x, \bar v) \in \gph \partial g$ with $\bar x \in \intr X$ a point at which $g$ is locally lsc and prox-regular with parameter $r > 0$. Let $\bar y$ be a point for which $\nabla_x \Phi(\bar x,\bar y)= \bar v$.
    Further assume that $-\Phi(\cdot, y)$ is strongly convex in the first argument for any $y$ near $\bar y$ with parameter $\delta \geq r$. Then $g$ is $\Phi$-prox-regular at $\bar x$ for $(\bar v, \bar y)$. 
\end{proposition}
\begin{proof}
Let $x,x'$ near $\bar x$ and $v \in \partial g(x)$ near $\bar v$ and $y$ near $\bar y$ such that $v = \nabla_x \Phi(x,y)$.
By strong convexity of $-\Phi(\cdot, y)$ we have that
$$
-\Phi(x',y) \geq -\Phi(x,y) - \langle v, x' - x \rangle + \tfrac{\delta}{2}\|x'-x\|^2.
$$
Since $g$ is prox-regular at $\bar x$ for $\bar v$ we also have
$$
g(x') \geq g(x) +\langle v, x'-x \rangle - \tfrac{r}{2}\|x' - x\|^2.
$$
Summing the two inequalities we obtain for $\delta \geq r$:
$$
g(x') - \Phi(x',y) \geq g(x) - \Phi(x,y) + \tfrac{\delta -r}{2}\|x'-x\|^2 \geq g(x) - \Phi(x,y).
$$
Reordering we obtain
$$
g(x') \geq g(x) + \Phi(x',y)  - \Phi(x,y),
$$
which concludes the proof.
\end{proof}

Next we provide sufficient conditions for the local single-valuedness of $(\partial_\Phi g)^{-1}$ and the continuous differentiability of $g^\Phi$ generalizing \cite[Theorem 2.14]{laude2021lower}. We will utilize the concept of a \emph{graphical localization}, a localization not only in domain but also in range. We next provide its definition, adapted from \cite[p. 3]{dontchev2009implicit}.
\begin{definition}[Graphical localization]
    For $F:\bR^n \rightrightarrows \bR^m$ and a pair $(\bar x, \bar v) \in \gph F$, a graphical localization of $F$ at $\bar x$ for $\bar v$ is a set-valued mapping $T$ such that $\gph T = (U \times V)\cap \gph F$ for some neighborhoods $U \in \mathcal{N}(\bar x)$ and $V \in \mathcal{N}(\bar v)$, so that 
    \begin{align}
        T(x) := \begin{cases}
    F(x)\cap V &\text{ if } x \in U \\
    \emptyset &\text{ otherwise}.
        \end{cases}
    \end{align}
    The inverse of $T$ then has
    \begin{align}
        T^{-1}(v) := \begin{cases}
    F^{-1}(v)\cap U &\text{ if } v \in V \\
    \emptyset &\text{ otherwise}.
        \end{cases}
    \end{align}
    and is thus a graphical localization of the set-valued mapping $F^{-1}$ at $\bar v$ for $\bar x$. By a single-valued localization of $F$ at $\bar x$ will be meant a graphical localization that is a function, its domain not necessarily being a neighborhood of $\bar x$. The case where the domain is indeed a neighborhood of $\bar x$ will be indicated by referring to a single-valued localization of $F$ around $\bar x$ for $\bar v$ instead of just at $\bar x$ for $\bar v$. 
\end{definition}
An example of a set-valued mapping that does not have a single-valued localization around a point can be constructed as in \cite[p. 6]{dontchev2009implicit}: for $F(x) = x^2$, $F^{-1}$ is single-valued at $0$, empty-valued for $y < 0$ and two-valued for $y > 0$. Clearly thus $F^{-1}$ does not have a single-valued localization around $0$ for $0$. The following special case of a localization defines a localization of $\partial g$ in the $g$-attentive topology \cite[Equation 8(2)]{RoWe98}.
\begin{definition}[$g$-attentive $\varepsilon$-localization of limiting subdifferential]
        Let $g:X \to \exR$ be locally lsc at $\bar x$, a point where $g(\bar x)$ is finite. Let $\bar v \in \partial g(\bar x)$. Then for some $\varepsilon > 0$ the $g$-attentive $\varepsilon$-localization $T_\varepsilon$ of $\partial g(\bar x)$ at $\bar x$ for $\bar v$ is defined by 
    \begin{align} \label{eq:DCA_primal}
        T_\varepsilon(x) = \begin{cases}
    \{v \in \partial g(x) \mid \|v-\bar v\|<\varepsilon\} &\text{ if } \|x-\bar x\|<\varepsilon \text{ and } g(x) < g(\bar x) + \varepsilon \\
    \emptyset &\text{ otherwise}.
        \end{cases}
    \end{align}
\end{definition}
If in addition the function is subdifferentially continuous \cite[Definition 13.28]{RoWe98}, the $g$-attentive localization of the limiting subdifferential coincides with its graphical localization. We can now move on to our local differentiability result for the $\Phi$-conjugates.
\begin{proposition}[Local continuous differentiability of $\Phi$-conjugates] \label{thm:equivalence_phi_prox_reg}
    Let \cref{assum:g} hold true and $g$ be finite at $\bar x \in \intr X$ with $\{\bar x\} = (\partial_\Phi g)^{-1}(\bar y)$ for some $\bar y \in Y$. Let $\nabla_x \Phi(\bar x, \bar y)=:\bar v$. Then we have $\bar v \in \partial g(\bar x)$ and the following statements are equivalent:
   \begin{propenum}
       \item \label{thm:equivalence_phi_prox_reg:prox_reg} $g$ is $\Phi$-prox-regular at $\bar x$ for $(\bar v,\bar y)$ such that the $\Phi$-prox-regularity inequality \cref{eq:phi_prox_regularity} holds strictly for $x' \neq x$ with $x',x$ near $\bar x$ and $y$ near $\bar y$.
       \item \label{thm:equivalence_phi_prox_reg:monotonicity} for all $(x,y),(x',y')$ near $(\bar x, \bar y)$ with $x\neq x'$ such that $\nabla_x \Phi(x,y) \in T_\varepsilon(x)$ and $\nabla_x \Phi(x',y') \in T_\varepsilon(x')$ the following strict $\Phi$-monotonicity property holds true:
       \begin{align}
           0 > \Phi(x', y) - \Phi(x,y) + \Phi(x,y')-\Phi(x',y'),
       \end{align}
       where $T_\varepsilon$ is a $g$-attentive $\varepsilon$-localization of $\partial g$ at $\bar x$ for $\bar v$.
       \item \label{thm:equivalence_phi_prox_reg:single_valued} $(\partial_\Phi g)^{-1}$ is a single-valued map in a neighborhood of $\bar y$ such that for any $y$ in that neighborhood we have:
       \begin{equation}
           (\partial_\Phi g)^{-1}(y) = \{x \in X \mid 0 \in T_\varepsilon(x)-\nabla_x \Phi(x, y)\},
       \end{equation}
       where $T_\varepsilon$ is a $g$-attentive $\varepsilon$-localization of $\partial g$ at $\bar x$ for $\bar v$.
   \end{propenum} 
   Futhermore, under any of the above equivalent conditions, $g^\Phi$ is continuously differentiable in a neighborhood of $\bar y$ where 
   \begin{equation} \label{eq:env_grad}
       \nabla g^\Phi(y) = \nabla_y \Phi((\partial_\Phi g)^{-1}(y), y),
   \end{equation}
   holds for any $y$ near $\bar y$.
\end{proposition}
\begin{proof}
        Firstly, note that since $\bar x = \argmin_{x \in X} \{ g(x) - \Phi(x, \bar y) \}$ we have by the nonsmooth version of Fermat's theorem \cite[Theorem 10.1]{RoWe98}, the fact that $\Phi(\cdot,\bar y)$ is continuously differentiable around $\bar x \in \intr X$ and the subgradient sum rule \cite[Exercise 10.10]{RoWe98}, that $0 \in \partial g(\bar x) - \nabla_x \Phi(\bar x, \bar y)$ and hence $\nabla_x \Phi(\bar x, \bar y) = \bar v \in \partial g(\bar x)$.

    ``\labelcref{thm:equivalence_phi_prox_reg:prox_reg} $\Rightarrow$ \labelcref{thm:equivalence_phi_prox_reg:monotonicity}'': Let $\varepsilon>0$ and $T_\varepsilon$ be the $g$-attentive $\varepsilon$-localization of $\partial g$ at $\bar x$ for $\bar v$. Let $(x,y),(x',y')$ near $(\bar x, \bar y)$ such that $\nabla_x \Phi(x,y) \in T_\varepsilon(x)$ and $\nabla_x \Phi(x',y') \in T_\varepsilon(x')$. For $\varepsilon$ small enough, by assumption, the $\Phi$-prox-regularity inequality \cref{eq:phi_prox_regularity} holds strictly at $(x,y),(x',y')$, i.e.,
    \begin{equation}
        g(x') > g(x) +\Phi(x', y) -\Phi(x, y) \quad \text{and} \quad g(x) > g(x') +\Phi(x, y') -\Phi(x', y').
    \end{equation}
    Summing the inequalities we obtain
    \begin{align}
        0 > \Phi(x', y) -\Phi(x, y)+ \Phi(x, y') -\Phi(x', y'),
    \end{align}
    as claimed.

    ``\labelcref{thm:equivalence_phi_prox_reg:monotonicity} $\Rightarrow$ \labelcref{thm:equivalence_phi_prox_reg:single_valued}'': Take any $y^\nu \to \bar y$ contained in $Y$ and $x^\nu \in \argmin_{x \in X} \{ g(x) - \Phi(x,y^\nu) \}$. In light of \cref{thm:conj_cont:osc}, since $\{\bar x\} = (\partial_\Phi g)^{-1}(\bar y)$, up to extracting a subsequence, $x^\nu \to \bar x$. By Fermat's rule and since $\ran (\partial_\Phi g)^{-1} \subseteq \intr X$ this implies that
    \begin{equation} \label{eq:optimality_proof}
        0 \in \partial g(x^\nu) - \nabla_x \Phi(x^\nu, y^\nu), 
    \end{equation}
    i.e., there is $\nabla_x \Phi(x^\nu, y^\nu) =v^\nu \in \partial g(x^\nu)$ such that by continuity of $\nabla_x \Phi$, $v^\nu \to \bar v$. Furthermore, since $g(x^\nu) - \Phi(x^\nu, y^\nu) = \inf_{x \in X} \{ g(x) - \Phi(x,y^\nu) \} \to \inf_{x \in X} \{ g(x) - \Phi(x, \bar y) \} = g(\bar x) - \Phi(\bar x, \bar y)$ and $\Phi(x^\nu, y^\nu) \to \Phi(\bar x, \bar y)$ we also have that $g(x^\nu) \to g(\bar x)$. Thus, for $y$ sufficiently near $\bar y$, we can replace $\partial g$ in \cref{eq:optimality_proof} with $T_\varepsilon$ and obtain
    $$
    \emptyset \neq \argmin_{x \in X} \{g(x) - \Phi(x,y) \} \subseteq \{x \in X \mid \nabla_x \Phi(x,y) \in T_\varepsilon(x)\}.
    $$
    Suppose that $x,x' \in \{x \in X : \nabla_x \Phi(x,y) \in T_\varepsilon(x)\}$ with $x\neq x'$. By \labelcref{thm:equivalence_phi_prox_reg:monotonicity} we have that
    $$
    0 > \Phi(x', y) - \Phi(x,y) + \Phi(x, y) - \Phi(x',y) = 0,
    $$
    a contradiction. Hence $x=x'$ and consequently $y \mapsto \argmin_{x \in X} \{g(x) - \Phi(x,y) \}$ is single-valued.

    ``\labelcref{thm:equivalence_phi_prox_reg:single_valued} $\Rightarrow$ \labelcref{thm:equivalence_phi_prox_reg:prox_reg}'': Choose $(x,v) \in \gph T_\varepsilon$ such that $v = \nabla_x \Phi(x,y)$ with $y$ sufficiently close to $\bar y$. Then we have that $0 \in T_\varepsilon(x) - \nabla_x \Phi(x,y)$ and by \labelcref{thm:equivalence_phi_prox_reg:single_valued} $\{x\} =\argmin_{x' \in X} g(x') - \Phi(x',y)$. In particular this implies that for all $x' \neq x$:
    $$
    g(x') - \Phi(x',y) > g(x) - \Phi(x,y),
    $$
    and hence in particular \labelcref{thm:equivalence_phi_prox_reg:prox_reg} which holds even globally for all $x' \neq x$.

    Regarding the continuous differentiability of $g^\Phi$, we will first show that $g^\Phi$ is lower-$\mathcal{C}^1$ \cite[Definition 10.29]{RoWe98} and then utilize the subsmoothness calculus \cite[Theorem 10.31]{RoWe98} along with the single-valuedness of the generalized prox $(\partial_\Phi g)^{-1}$ in a neighborhood of $\bar y$.

    Let $y_0 \in Y$ and $V \subset Y$ a compact neighborhood of $y_0$. Let, moreover, $h(x,y) := \Phi(x,y)-g(x)$ and define the set $Z = \{x \in (\partial_\Phi g)^{-1}(y) \mid y \in V\}$. First, we show that $Z$ is closed. Assume the contrary and we have a sequence $\{x^\nu\}_{\nu \in \bN} \subset Z$ with $x^\nu \to x^\star \notin Z$. This implies the existence of a sequence $\{y^\nu\}_{\nu \in \bN}$, $y^\nu \in V$ for all $\nu \in \bN$. Since $\{y^\nu\}$ is bounded, as $V$ is bounded, there exists a convergent subsequence $y^{\nu_j} \to y^\star \in V$ since $V$ is also closed. Then, from \cref{thm:conj_cont:osc} we have that $x^{\nu_j} \to x^\star \in (\partial_\Phi g)^{-1}(y^\star) \subset Z$, a contradiction. Now assume that $Z$ is not bounded. Then, there exists a $\{x^\nu\}_{\nu \in \bN} \subset Z$ such that $\|x^\nu\| \to \infty$ and thus also a $\{y^\nu\}_{\nu \in \bN} \subset V$. Since $\{y^\nu\}_{\nu \in \bN}$ is bounded, there exists a convergent subsequence $y^{\nu_j} \to y^\star \in V$, which then implies that $\{x^{\nu_j}\}$ is bounded contradicting the fact that $\|x^\nu\| \to \infty $. Therefore, $Z$ is compact.

    By definition, $\argmax_{x \in Z}h(x,y) = (\partial_\Phi g)^{-1}(y)$ and $g^\Phi(y) = \max_{x \in Z}h(x, y)$. Note that since $\nabla_y h(x,y) = \nabla_y \Phi(x, y)$ depends continuously on $(x, y)$ we only have to show the continuity of $g$ on the set $Z$ in order to show the continuity of $h$ w.r.t. $(x,y)$ and thus obtain the claimed subsmoothness of $g^\Phi$.

    This will follow through a contradiction argument. Assume that $g$ is not continuous on $Z$. Then, there exists a sequence $\{x^\nu\}_{\nu \in \bN} \in Z$ that converges to some $x^\star \in Z$ such that $g(x^\nu) \nrightarrow g(x^\star)$. First, note that $g$ being closed implies that it is bounded from below in the compact set $Z$, i.e. $-\infty < \delta \leq g(x^\nu)$ for all $\nu \in \bN$. Moreover, for every $x^\nu \in Z$ there exists by definition a $y^\nu \in V$ such that $x^\nu \in (\partial_\Phi g)^{-1}(y^\nu)$ and $g^\Phi(y^\nu) = \Phi(x^\nu, y^\nu) - g(x^\nu)$. Since, by \cref{thm:conj_cont:conj_cont}, $g^\Phi$ is a continuous function, it is bounded in the compact set $V$ and the same holds for $\Phi$ in $Z \times V$. Therefore, $g(x^\nu)$ is also bounded from above, which further implies that the sequence $\{g(x^\nu)\}_{\nu \in \bN}$ is bounded and since $g(x^\nu) \nrightarrow g(x^\star)$, there exists at least one subsequence let's say indexed by $\nu_j$ such that $g(x^{\nu_j}) \to g^\star \neq g(x^\star)$. By extracting another subsequence if necessary, in light of \cref{thm:conj_cont:osc} we can assume that $y^{\nu_j} \to y^\star$ with $x^\star \in (\partial_\Phi g)^{-1}(y^\star)$. Now, since $g^\Phi$ is continuous we have that $g^\Phi(y^{\nu_j}) \to g^\Phi(y^\star)$ which as $\Phi$ is continuous implies that $g(x^{\nu_j}) \to g(x^\star)$, a contradiction.
    Hence, $g$ is continuous on $Z$. Therefore, for any $y_0 \in Y$ there exists a neighborhood $V$ such that $g^\Phi(y) = \max_{x \in Z}h(x,y)$ where $Z$ is a compact set and $h$ and $\nabla h_y$ depend continuously on $(x,y) \in Z \times V$. This means that $g^\Phi$ is lower-$\cC^1$.

    In light of \cite[Theorem 10.31]{RoWe98}, we have that 
    \begin{align*}
        \partial g^\Phi(y)
        &=
        \con\{-\nabla_y \Phi(x, y) \mid x \in (\partial_\Phi g)^{-1}(y)\}
        \\
        &=
        -\nabla_y \Phi((\partial_\Phi g)^{-1}(y), y),
    \end{align*}
    where the second equality follows by the single-valuedness of $(\partial_\Phi g)^{-1}$ for $y$ sufficiently close to $\bar y$. 
\end{proof}

\renewcommand{\arraystretch}{1.5}
\begin{table}[t]
\centering
\begin{tabularx}{\textwidth}{ |Y|c|c|Y|Y| }  
 \hline 
     Name & $X$ & $Y$ & $\Phi(x, y)$ & $\nabla g^\Phi(y)$ \\ 
  \hline \hline 
 Euclidean & $\bR^n$ & $\bR^n$ & $-\tfrac{1}{2\gamma}\|x-y\|^2$ 
 & $\tfrac{1}{\gamma}(P(y)-y)$\\ 
 \hline 
 left Bregman & $\dom h$ &$\intr \dom h$ & $-\tfrac{1}{\gamma}D_h(x,y)$ 
 & $\tfrac{1}{\gamma}\nabla^2 h(y)(P(y)-y)$ \\ 
 \hline 
 right Bregman &$\bR^n$ &$\bR^n$ & $-\tfrac{1}{\gamma}D_h(y,x)$   & $\tfrac{1}{\gamma}(\nabla h(P(y))-\nabla h(y))$\\ 
 \hline
 Anisotropic &$\bR^n$ &$\bR^n$ & $-\gamma \star \phi(x-y)$   & $\nabla \phi(\tfrac{1}{\gamma}(P(y)-y))$\\ 
 \hline
  Entropic &$\bR^n_{+}$ &$\bR^n_{++}$ & $-\gamma d_\varphi(x,y)$   & $\gamma \varphi^*(\varphi'(\tfrac{P_i(y)}{y_i}))$\\ 
 \hline 
\end{tabularx}
\caption{ \label{tab:envs}
    Examples of $\Phi$-conjugates (generalized envelopes) and their gradients as obtained from \cref{thm:equivalence_phi_prox_reg}. For ease of exposition, we denote by $P(y):= (\partial_\Phi g)^{-1}(y)$ the generalized prox. Note that the sign flip compared to the standard versions of the gradients is due to $g^\Phi(y) = -\inf_{x \in X} g(x)-\Phi(x,y)$. The first line corresponds to the (negative) standard (Euclidean) Moreau envelope \cite[Proposition 13.37]{RoWe98}. The second line corresponds to the (negative) left Bregman--Moreau envelope \cite{kan2012moreau,laude2020bregman,ahookhosh2021bregman,laude2021lower,laude2023dualities,wang2025bregman} and the third one to the (negative) right Bregman--Moreau envelope \cite{bauschke2018regularizing,laude2020bregman,laude2021lower,wang2025bregman}. In the fourth line the (negative) anisotropic Moreau envelope is presented \cite{moreau1966fonctionnelles,combettes2013moreau,laude2019optimization,laude2021lower,laude2023dualities,laude2025anisotropic}. The fifth line corresponds to the (negative) entropic envelopes from \cite{teboulle1992entropic,iusem1994entropy} and the form of the gradient to the $i$-th coordinate.
    }
\end{table}
When specifying the coupling functions $\Phi$, \cref{thm:equivalence_phi_prox_reg} retrieves the form of the gradients of standard envelope functions thus generalizing various existing results, but also leading to new ones. We present gradients of envelope functions from the related literature in \cref{tab:envs}.

\subsection{Lipschitz Differentiability and \texorpdfstring{$\mathcal{C}^2$}{C2} properties of \texorpdfstring{$\Phi$}{Phi}-conjugates} \label{subsec:lipschitz_diff_conj}
In the previous subsection we proved the single-valuedness of the proximal mapping which entails the continuous differentiability of the $\Phi$-conjugate under (a generalization of) prox-regularity. In this subsection our aim is to establish a much stronger property: The Lipschitz continuity of the proximal mapping which entails the Lipschitz differentiability and eventually the strict twice differentiability of the $\Phi$-envelope under stronger conditions. In particular the latter requires $\Phi$ to be twice differentiable and $g$ to be strictly twice epi-differentiable \cite[Section 6]{poliquin1996prox}.  Note that twice epi-differentiability and its strict counterpart have been extensively studied and are considered standard in the literature \cite{rockafellar1985maximal,rockafellar1989proto,poliquin1996prox, poliquin1996generalized, mohammadi2020twice,hang2025smoothness,hang2024chain, hang2024role,gfrerer2025strict}.

While Lipschitz continuity of the proximal mapping is easy to obtain for the Bregman case \cite[Corollary 2.43]{laude2021lower} the $\Phi$-convex case requires a more powerful technique:
Following \cite{laude2019optimization,laude2021lower} our key idea is to invoke the implicit function theorem for generalized equations \cite{dontchev2009implicit} originally due to Robinson \cite{robinson1991implicit} identifying the proximal mapping (locally) with its resolvent that we have introduced in \cref{thm:equivalence_phi_prox_reg:single_valued}. In the context of the implicit function theorem we interpret the resolvent in terms of (a graphical localization of) the solution mapping:
\begin{equation} \label{eq:solution_map}
    S: y \mapsto \{x \in \bR^n \mid 0 \in \partial g(x)-\nabla_x \Phi(x, y) \}.
\end{equation}
It is important to stress that in the full generality of our setting where neither $g$ or $-\Phi(\cdot,y)$ are considered convex, $S(y)$ can in general be a much larger set than $(\partial_\Phi g)^{-1}(y)$ at any point $y \in Y$. We next state the main result of this section that we prove in multiple steps afterwards.

\begin{theorem}[Lipschitz continuity and strict differentiability of the generalized prox] \label{thm:single_valued_prox} Let \cref{assum:g} hold and $g$ be prox-regular at $\bar x \in \intr X$ for $\bar v \in \partial g(\bar x)$ with constant $r$. Let $\bar y \in Y$ and assume that $(\partial_\Phi g)^{-1}(\bar y) = \{\bar x\}$. Let $\Phi$ be twice continuously differentiable such that $\lambda_{\max}(M) < r^{-1}$ for $M := -\nabla_{xx}^2 \Phi(\bar x, \bar y)^{-1}$.
Then the following statements are true:
    \begin{propenum}
        \item \label{thm:single_valued_prox:lipschitz} The mapping
        $$
        (\partial_\Phi g)^{-1} : y \mapsto \{x \in \bR^n \mid 0 \in T_\varepsilon(x)-\nabla_x \Phi(x, y) \}
        $$ is single-valued and Lipschitz continuous in a neighborhood of $\bar y$, where $T_\varepsilon$ is a $g$-attentive $\varepsilon$-localization of $\partial g$.
        
        \item \label{thm:single_valued_prox:diff} If, moreover, $g$ is strictly twice epi-differentiable at $\bar x$ for $\bar v$, then $(\partial_\Phi g)^{-1}$ is strictly differentiable at $\bar y$ with Jacobian 
        \begin{align}
        \nabla (\partial_\Phi g)^{-1}(\bar y)&= \nabla (T_\varepsilon +M^{-1})^{-1}\big(\nabla_x \Phi(\bar x, \bar y) +M^{-1} \bar x\big) \nabla_{xy}^2\Phi(\bar x, \bar y).
        \end{align}

        \item \label{thm:single_valued_prox:diff_prox_bound} 
        If, in addition, $g$ is prox-bounded and $\lambda_{\max}(M)$ is small enough we can further rewrite the Jacobian as
        \begin{equation}
            \nabla (\partial_\Phi g)^{-1}(\bar y)
            = M \nabla \prox{M}{g}(\bar x +M\nabla_x \Phi(\bar x, \bar y))\nabla_{xy}^2\Phi(\bar x, \bar y)
        \end{equation}
        where $\prox{M}{g}(x):= \argmin_{x\in \bR^n} \{\frac{1}{2}\langle x- y, M^{-1}(x-y) \rangle + g(x) \}$ is the scaled proximal mapping.
    \end{propenum}
\end{theorem}
\cref{thm:single_valued_prox:lipschitz} was already proven in \cite[Proposition 2.9]{laude2021lower}. In the following we provide a full proof of the complete statement. To this end we consider a linearization of the optimality condition in $S$:
\begin{equation} \label{eq:linearized}
    Q(x) := H(x) + \partial g(x),
\end{equation}
where $H(x):=-\nabla_x \Phi(\bar x, \bar y) - \nabla_{xx}^2 \Phi(\bar x, \bar y) (x-\bar x)$ is a first-order Taylor expansion of $-\nabla_x \Phi(\cdot, \bar y)$ around $\bar x$. In order to prove \cref{thm:single_valued_prox} we first require some intermediate results. Our first result shows that indeed $Q^{-1}$ has a Lipschitz continuous single-valued localization around $0$ for $\bar x$, which then implies the same for our mapping of interest $S$ through \cite[Corollary 2B.10]{dontchev2009implicit}.
\begin{lemma} \label{thm:lipschitz_sing_val_loc}
    Under the assumptions of \cref{thm:single_valued_prox} consider the solution mapping $S$ as defined in \eqref{eq:solution_map}, a pair $(\bar y, \bar x) \in \gph S$ and $\bar v = \nabla_x \Phi(\bar x, \bar y) \in \partial g(\bar x)$. Then, the inverse mapping $Q^{-1}$ of $Q$ in \eqref{eq:linearized} has a Lipschitz continuous single-valued localization $\sigma$ around $0$ for $\bar x$ with 
        \[
        \sigma(u) =(T_\varepsilon+M^{-1})^{-1}(u + \bar v +M^{-1} \bar x),
        \]
    where $T_\varepsilon$ is a $g$-attentive $\varepsilon$-localization of $\partial g$ at $\bar x$ for $\bar v = -\nabla_x \Phi(\bar x, \bar y)$ and $M = -\nabla_{xx}^2 \Phi(\bar x, \bar y)^{-1}$. This implies the existence of a Lipchitz continuous, single-valued localization $s$ of $S$ around $\bar y$ for $\bar x$.
    If, furthermore, $g$ is prox-bounded and $ \lambda_{\max}(M)$ is small enough, we have the identity $(T_\varepsilon+M^{-1})^{-1} = \prox{M}{g} \circ M$. 
\end{lemma}
\begin{proof}
    Let us start by considering the following mapping:
    \begin{equation} \label{eq:tilde_G}
        \widetilde{Q}(x) =H(x) + T_\varepsilon(x)= -\nabla_x \Phi(\bar x, \bar y) - \nabla_{xx}^2 \Phi(\bar x, \bar y) (x-\bar x) + T_\varepsilon(x),
    \end{equation}
    where $T_\varepsilon$ is a $g$-attentive $\varepsilon$-localization of $\partial g$ at $\bar x$ for $\bar v=\nabla_x \Phi(\bar x, \bar y)$ and $\varepsilon$ is the radius of the prox-regularity inequality. Our goal is to show that the inverse of $\sigma :=\widetilde{Q}^{-1}$ is a Lipschitz continuous single-valued localization of $Q^{-1}$ around $0$ for $\bar x$ which would then imply the existence of a single-valued localization of $S$ around $\bar y $ for $\bar x$ through \cite[Corollary 2B.10]{dontchev2009implicit}. We thus have to show first that the domain of $\sigma$ is actually a neighborhood of $0$. To that aim consider the function 
    \[
        h(x, u) := g(x) + \delta_{\overline{B}(\bar x, \varepsilon)}(x) - \tfrac{1}{2}\langle x - \bar x,\nabla_{xx}^2 \Phi(\bar x, \bar y)( x-\bar x) \rangle - \langle \nabla_x \Phi(\bar x, \bar y),x \rangle - \langle x, u \rangle,
    \]
    for $\varepsilon > 0$ such that the prox-regularity inequality holds in $x \in \overline{B}(\bar x, \varepsilon)$ and notice that it is level-bounded in $x$ locally uniformly in $u$. Moreover, since $\overline{B}(\bar x, \varepsilon)$ is closed, $\delta_{\overline{B}(\bar x, \varepsilon)}(x)$ is lsc and since $g$ is proper and lsc, their sum is also proper and lsc. Therefore, $h$ is also proper and lsc as the rest of the terms in its definition are continuous in $(x, u)$. Now, observe that $\bar v \in \partial_x h(\bar x, 0)$ and $h(\bar x,0) = g(\bar x) - \langle \nabla_x \Phi(\bar x, \bar y), \bar x \rangle$ and from prox-regularity of $g$ we have for all $x \in \overline{B}(\bar x, \varepsilon)$,
    \begin{equation*}
        g(x) \geq g(\bar x) + \langle \bar v,x-\bar x \rangle - \tfrac{r}{2}\|x-\bar x\|^2.
    \end{equation*}
    The above can be rewritten as
    \begin{equation} \label{eq:local_prox_reg}
        g(x) + \delta_{\overline{B}(\bar x, \varepsilon)}(x)\geq g(\bar x) + \langle \bar v,x-\bar x \rangle - \tfrac{r}{2}\|x-\bar x\|^2 \qquad \forall x \in X.
    \end{equation}
    Moreover, the function $p(x) := - \tfrac{1}{2}\langle x-\bar x ,\nabla_{xx}^2 \Phi(\bar x, \bar y) (x-\bar x) \rangle - \langle \nabla_x \Phi(\bar x, \bar y),x \rangle$ is strongly convex with parameter $\lambda_{\max}(M)^{-1}$, which implies that 
    $$
    p(x) \geq p(\bar x) + \langle \nabla p(\bar x),x-\bar x \rangle + \tfrac{\lambda_{\max}(M)^{-1}}{2}\|x-\bar x\|^2.
    $$
    Note that $\nabla p(\bar x) = - \nabla_x \Phi(\bar x, \bar y) = -\bar v$ by assumption and thus by summing the two inequalities we obtain:
    \begin{equation*}
        h(x, 0) \geq h(\bar x,0) + \tfrac{ \lambda_{\max}(M)^{-1}-r}{2}\|x-\bar x\|^2.
    \end{equation*} 
    This inequality implies that $h(x,0) > h(\bar x, 0)$ for all $x \neq \bar x$, which in turn implies that $\argmin_x h(x,0) = \{\bar x\}$.
    
    In light of \cite[Theorem 1.17]{RoWe98} we have that for $u^\nu \to 0$ such that $\inf_x h(x, u^\nu) < \infty$ there is, up to extracting a subsequence, $x^\nu \in \argmin_x h(x,u^\nu) \neq \emptyset $ with $x^\nu \to \bar x$ and $\inf_x h(x, u^\nu) \to \inf_x h(\bar x, 0)$.

    Now, since $x^\nu \to \bar x$ for $\nu$ large enough we can remove the normal cone of $\overline{B}_{(\bar x, \varepsilon)}$ from the optimality conditions for this minimization problem and get 
    $$u^\nu + \nabla_x \Phi(\bar x, \bar y) + \nabla_{xx}^2 \Phi(\bar x, \bar y) (x^\nu-\bar x) \in \partial g(x^\nu).
    $$ 
    Since the term on the l.h.s. is continuous in $x$ and $u^\nu \to 0$, we have that for $\nu$ large enough, $\|u^\nu + \nabla_x \Phi(\bar x, \bar y) + \nabla_{xx}^2 \Phi(\bar x, \bar y) (x^\nu-\bar x) - \bar v\| \leq \varepsilon$, for $\varepsilon$ as in the prox-regularity definition, meaning that $\|v^\nu - \bar v\| \leq \varepsilon$ for $v^\nu \in \partial g(x^\nu)$. Moreover, $\inf_x h(x, u^\nu) \to \inf_x h(x, 0)$ and since all the other terms in $h$ are continuous, $g(x^\nu) \to g(\bar x)$. Thus we can replace the subdifferential $\partial g$ of $g$ with the $g$-attentive $\varepsilon$-localization $T_\varepsilon$ and since the above holds for any $u^\nu \to 0$, we have for any $u$ in a sufficiently small neighborhood of $0$:
    \begin{equation*}
        u \in \widetilde{Q}(x),
    \end{equation*}
    for some $x$ near $\bar x$. Therefore, $\dom \sigma = \dom \widetilde{Q}^{-1}$ is a neighborhood of $0$.
    
    Next we show that $\widetilde Q$ is strongly monotone.
    To that aim, pick any $(x_1, v_1), (x_2, v_2) \in \gph T_\varepsilon$. From prox-regularity of $g$ it follows that
    \begin{align*}
        & g(x_1) \geq g(x_2) + \langle v_2,x_1-x_2 \rangle - \tfrac{r}{2}\|x_1-x_2\|^2 \\
        & g(x_2) \geq g(x_1) + \langle v_1,x_2-x_1 \rangle - \tfrac{r}{2}\|x_1-x_2\|^2.
    \end{align*}
    and summing the above inequalities we obtain
    \begin{equation}
        \langle v_1-v_2,x_1-x_2 \rangle \geq -r\|x_1-x_2\|^2.
    \end{equation}
    Adding and subtracting $-\langle \nabla_{xx}^2 \Phi(\bar x, \bar y)(x_1-x_2),x_1 - x_2 \rangle$ to the inequality above, we obtain
    \begin{align*}
        \langle u_1-u_2,x_1-x_2 \rangle 
        & \geq -r\|x_1-x_2\|^2 -\langle \nabla_{xx}^2 \Phi(\bar x, \bar y)(x_1-x_2),x_1 - x_2 \rangle 
        \\
        & \geq (\lambda_{\max}(M)^{-1}-r)\|x_1-x_2\|^2,
    \end{align*}
    for $u_1 \in \widetilde{Q}(x_1)$ and $u_2 \in \widetilde{Q}(x_2)$, where the second inequality follows by the bound on the eigenvalues of $\nabla_{xx}^2 \Phi(\bar x, \bar y)$. Therefore, $\widetilde{Q}$ is a $(\lambda_{\max}(M)^{-1}-r)$-strongly monotone localization of $Q$ at $\bar x$ for $0$.
    
    Since $\widetilde Q$ is strongly monotone, $\sigma = \widetilde{Q}^{-1}$ is Lipschitz continuous on $\ran \widetilde{Q} = \dom \sigma$ via \cite[Lemma 1.10]{laude2021lower}. Therefore, $\sigma$ is a Lipschitz continuous, single-valued localization of $Q^{-1}$ around $0$ for $\bar x$. The formula for $\sigma$ follows now directly by \eqref{eq:tilde_G}, noting that $u \in \widetilde Q(x) \Longleftrightarrow x \in \widetilde{Q}^{-1}(u)$.

    Now, consider the case where $g$ is prox-bounded with constant $\lambda_g$. Through \cite[Proposition 8.46(f)]{RoWe98}, we can find $\lambda \in (0,\lambda_g)$ such that $\lambda^{-1} > r$ and the inequality \eqref{eq:local_prox_reg} can be written globally, i.e.
    \begin{equation*}
        g(x) \geq g(\bar x) + \langle \bar v,x-\bar x \rangle - \tfrac{1}{2\lambda}\|x-\bar x\|^2 \qquad \forall x \in X.
    \end{equation*}
    Therefore, for $\lambda_{\max}(M)$ small enough such that $(\lambda_{\max}(M))^{-1} > \lambda^{-1}$ we can repeat the same analysis as above, removing the indicator of $\overline{B}(\bar x, \varepsilon)$. The form of the inverse mapping follows by the identification with $\argmin_x h(x,u)$.

    Finally, the existence of the Lipschitz continuous localization $s$ follows from \cite[Corollary 2B.10]{dontchev2009implicit} since $\Phi$ is twice continuously differentiable and thus $\nabla_x \Phi$ is strictly differentiable at $(\bar x, \bar y)$.
\end{proof}
So far we have shown the Lipschitz continuity of the localization $s$ of $(\partial_\Phi g)^{-1}$ in a neighborhood of $\bar y$. Therefore, in order to prove the first part of \cref{thm:single_valued_prox} we only have to identify $(\partial_\Phi g)^{-1}$ with its localization around $\bar y$. The Lipschitz continuity would then imply the a.e. differentiability of this mapping. However, we are interested in \emph{strict} differentiability at specific points which in general might not hold, even for globally Lipschitz functions as showcased in \cite{daniilidis2024all}. 

In order to prove the aforementioned strict differentiability we will rely on the approach of \cite{dontchev2009implicit} that builds upon the notion of first-order approximation.
\begin{definition}[First-order approximation] \label{def:fo_appr}
    Consider a function $F: \bR^n \to \bR^m$ and a point $\bar x \in \intr \dom F$. A function $H: \bR^n \to \bR^m$ is a first-order approximation to $F$ at $\bar x$ if $H(\bar x) = F(\bar x)$ and for every $\varepsilon > 0$ there exists $V \in \mathcal{N}(\bar x)$ such that 
    \begin{equation*}
        \|F(x) - H(x)\| \leq \varepsilon \|x-\bar x\| \qquad \forall x \in V.
    \end{equation*}
    It is a strict first-order approximation if 
    \begin{equation*}
        \|F(x)-H(x) - (F(x')-H(x'))\| \leq \varepsilon \|x-x'\| \qquad \forall x, x' \in V.
    \end{equation*}
\end{definition}
\begin{lemma} \label{thm:strict_fo_approx}
    In the setting of \cref{thm:lipschitz_sing_val_loc}, there exists a first-order approximation $\eta$ to $s$ at $\bar y$ given by
        \begin{equation}
            \eta(y) = \sigma (\nabla_{yx}^2 \Phi(\bar x,\bar y)(y-\bar y)).
        \end{equation}
    If $\sigma$ is strictly differentiable at $0$, then the approximation $\eta$ is strict.
\end{lemma}
\begin{proof}
    The fact that $\eta$ is a first-order approximation of $s$ follows immediately from \cite[Corollary 2B.10]{dontchev2009implicit}.

    Now assume moreover that $\sigma$ is strictly differentiable at $0$. We will work in a similar way as in the proof of \cite[Theorem 2B.9]{dontchev2009implicit}. Let $V \in \mathcal{N}(\bar y)$ be such that $s$ is single valued on $V$ and note that, by substitution, for $y \in V$ we have $s (y) = \sigma (-e(s(y), y))$ for $e(x, y) := -\nabla_x \Phi(x, y) + \nabla_x \Phi(\bar x, \bar y) + \nabla_{xx}^2 \Phi(\bar x, \bar y)(x - \bar x)$. Now take any $\varepsilon > 0$ and using the strict differentiability of $\nabla_x \Phi$ at $(\bar x, \bar y)$ we have for some $\overline V \in \mathcal{N}((s(\bar y),\bar y))$:
    \begin{align} \nonumber
        &\| \nabla_x \Phi(s_y,y)-\nabla_x \Phi(s_y',y') - \nabla_{xx}^2 \Phi(\bar x, \bar y)(s_y-s_y')
        - \nabla_{yx}^2 \Phi(\bar x,\bar y)(y-y')\|\\
        &\qquad\leq \varepsilon (\|s_y-s_y'\|+\|y-y'\|), \label{eq:eps_approx_2}
    \end{align}
    for all $(s_y,y), (s_y',y') \in \overline V$. Now, for all $y \in V$, from the Lipschitz continuity of $s$ with constant say $L_s$ we have that $\|s(y) - s(\bar y)\| \leq L_s\|y-\bar y\|$ or that $s(y)$ belongs to a ball around $s(\bar y)$ with radius that depends on $y$ and thus by potentially shrinking $V$ we can enforce all points $(s(y),y) \in \overline V$. Therefore, we can choose $V$ such that
    \begin{align} \nonumber
        &\| \nabla_x \Phi(s(y),y)-\nabla_x \Phi(s(y'),y') - \nabla_{xx}^2 \Phi(\bar x, \bar y)(s(y)-s(y'))
        - \nabla_{yx}^2 \Phi(\bar x,\bar y)(y-y')\|\\ 
        &\qquad\leq \varepsilon (\|s(y)-s(y')\|+\|y-y'\|), \label{eq:eps_approx}
    \end{align}
    for all $y, y' \in V$.
    Now note that the function $-e(s(y), y)$ is also locally Lipschitz as the composition of a locally Lipschitz (since $\Phi$ is assumed twice continuously differentiable) and a Lipschitz function, i.e. $\|e(s(y), y) - e(s(\bar y), \bar y))\| \leq \tilde L\|y-\bar y\|$ for any $y \in V$ and thus by making $V$ even smaller if needed we can use the strict differentiability of $\sigma$ at $0$ to get
    \begin{align} \nonumber
        &\|\sigma (-e(s(y), y)) - \sigma (-e(s(y'), y')) - \nabla \sigma(0)(-e(s(y), y) + e(s(y'), y'))\| \\ 
        &\qquad\leq \varepsilon \|e(s(y'), y') - e(s(y), y))\|
        \label{eq:bound_eps}
    \end{align}
    for all $y, y' \in V$. Similarly, 
    \begin{align}
        &\|\sigma (\nabla_{yx}^2 \Phi(\bar x,\bar y)(y'-\bar y)) - \sigma (\nabla_{yx}^2 \Phi(\bar x,\bar y)(y-\bar y)) - \nabla \sigma(0)\nabla_{yx}^2 \Phi(\bar x,\bar y)(y'-y)\| \leq \varepsilon \|y-y'\|.\label{eq:bound_sigma}
    \end{align}
    Now note that for any $y, y' \in V$, we have
    \begin{align} \nonumber
        &\|s(y)-s(y')+ \eta(y')-\eta(y)\| \\
        &\qquad=\|\sigma (-e(s(y), y)) - \sigma (-e(s(y'), y')) + \sigma (\nabla_{yx}^2 \Phi(\bar x,\bar y)(y'-\bar y)) -\sigma (\nabla_{yx}^2 \Phi(\bar x,\bar y)(y-\bar y))\|. \label{eq:norm_break}
    \end{align}
    In the following we are going to harness the strict differentiability of $\sigma$ to show that $\eta$ is a strict first-order approximation to $s$, i.e., that it satisfies \cref{def:fo_appr}. Note that 
    $$-e(s(y),y) + e(s(y'),y') = \nabla_x \Phi(s(y),y)-\nabla_x \Phi(s(y'),y') - \nabla_{xx}^2 \Phi(\bar x, \bar y)(s(y)-s(y'))
    $$
    and add and subtract $\nabla \sigma(0)(-e(s(y),y) + e(s(y'),y'))$ along with 
    $$\nabla \sigma(0)(\nabla_{yx}^2\Phi(\bar x, \bar y)(y'-y))$$ 
    to \eqref{eq:norm_break}. Using the triangle inequality we obtain the following estimate:
    \begin{align*}
        &\|s(y)-s(y')+ \eta(y')-\eta(y)\| 
        \\
        &\qquad\leq
        \|\sigma (-e(s(y), y)) - \sigma (-e(s(y'), y')) - \nabla \sigma(0)(-e(s(y), y) + e(s(y'), y'))\|
        \\
         &\qquad\quad+ \|\sigma (\nabla_{yx}^2 \Phi(\bar x,\bar y)(y'-\bar y)) - \sigma (\nabla_{yx}^2 \Phi(\bar x,\bar y)(y-\bar y)) - \nabla \sigma(0)\nabla_{yx}^2 \Phi(\bar x,\bar y)(y'-y)\|
        \\
         &\qquad\quad+ \Big\|\nabla \sigma(0)\big(\nabla_x \Phi(s(y),y)-\nabla_x \Phi(s(y'),y')  \\
         &\qquad\quad\quad- \nabla_{xx}^2 \Phi(\bar x, \bar y)(s(y)-s(y'))+ \nabla_{yx}^2 \Phi(\bar x,\bar y)(y'-y)\big)\Big\|.
    \end{align*}
    By using inequalities \eqref{eq:eps_approx}, \eqref{eq:bound_eps}, \eqref{eq:bound_sigma} along with the properties of the norm we obtain
    \begin{align*}
        &\|s(y)-s(y')+ \eta(y')-\eta(y)\| 
        \\
        &\qquad\leq \varepsilon \|\nabla_x \Phi(s(y),y)-\nabla_x \Phi(s(y'),y') - \nabla_{xx}^2 \Phi(\bar x, \bar y)(s(y)-s(y'))\|
        \\
        &\qquad\quad + \varepsilon \|\nabla_{yx}^2 \Phi(\bar x,\bar y)\| \|y-y'\| + \varepsilon \|\nabla \sigma(0)\| \|s(y)-s(y')\| +  \varepsilon \|\nabla \sigma(0)\| \|y-y'\|
    \end{align*}
    Further utilizing the Lipschitz continuity of $s$ on $V$ and of $\nabla_x \Phi$ as well as the fact that $\varepsilon$ is arbitrarily small we obtain that $\|s(y)-s(y')+ \eta(y')-\eta(y)\| \leq \tilde{\varepsilon} \|y-y'\|$ for arbitrary $\tilde{\varepsilon}$, which is the required result \cref{def:fo_appr}.
\end{proof}
In \cref{thm:strict_fo_approx} we showed that assuming $\sigma$ is strictly differentiable at $0$, there exists a strict first-order approximation to $s$ at $\bar y$. However, we have not shown under which conditions is $\sigma$ strictly differentiable. To that aim we require more regularity on $g$, namely the aforementioned twice epi-differentiability property and the closely related proto-differentiability of its subdifferential \cite{poliquin1996generalized}. Since these properties are geometric in the sense that they refer to the graphs of functions and mappings, we next present some standard tools and results from variational geometry.

In what follows, we use the definition and properties of set-limits as in \cite[Chapter 4]{RoWe98}. The \emph{tangent cone} at a point $\bar x$ of a set $C \subset \bR^n$ is given by 
\begin{equation} \label{eq:tangent_cone}
    T_{C}(\bar x) = \limsup_{t \searrow 0}\frac{C - \bar x}{t}
\end{equation}
and is composed of the vectors $w \in \bR^n$ such that there exists a sequence $x^\nu \xrightarrow[]{C} \bar x$ and $t^\nu \searrow 0$ with $(x^\nu - \bar x) / t^\nu \to w$, the notation $x^\nu \xrightarrow[]{C} \bar x$ meaning that $x^\nu \to \bar x$ with $x^\nu \in C$. Following \cite[Proposition 6.2]{RoWe98} we say that $C$ is \emph{geometrically derivable} at $\bar x$ if the $\limsup$ in \eqref{eq:tangent_cone} is actually a full limit, i.e., 
\begin{equation*}
    \limsup_{t \searrow 0}\frac{C - \bar x}{t} = \liminf_{t \searrow 0}\frac{C - \bar x}{t}.
\end{equation*}
The quantity on the r.h.s. of the equation above is known in the related literature as the \emph{derivative cone} of $C$ at $\bar x$ \cite{rockafellar1989proto} or the intermediate cone \cite{frankowska1985adjoint}.

The \emph{regular (Clarke) tangent cone} \cite[Definition 6.25]{RoWe98} and the \emph{(Boulingard) paratingent cone} to a set $C$ at $\bar x$ are defined, respectively,  as
\begin{equation}
    \widehat{T}_C(\bar x) = \liminf_{x \xrightarrow[]{C} \bar x, t \searrow 0} \frac{C -x}{t} \quad\text{ and }\quad \widetilde{T}_C(\bar x) = \limsup_{x \xrightarrow[]{C} \bar x, t \searrow 0} \frac{C -x}{t}.
\end{equation}
It is straightforward that
\begin{equation} \label{eq:cone_inclusion}
     \widehat{T}_C(\bar x) \subseteq T_{C}(\bar x) \subseteq \widetilde{T}_C(\bar x),
\end{equation}
while the inverse inclusions are related to notions of generalized differentiability. Following \cite[Corollary 6.29]{RoWe98} we say that a locally closed set $C \subset \bR^n$ is \emph{(Clarke) regular} at $\bar x \in C$ if $T_C(\bar x) = \widehat{T}_C(\bar x)$, i.e. if all tangent vectors at $\bar x$ are regular. Due to the inclusions above, the required condition then basically boils down to $\widehat{T}_C(\bar x) \supseteq T_{C}(\bar x)$. When $C$ is the graph of some set-valued mapping $T:\bR^n \rightrightarrows \bR^m$ we say that $T$ is \emph{graphically regular} \cite[Definition 8.38]{RoWe98} at $\bar x$ for $\bar u$ if $\gph T$ is Clarke regular at $(\bar x,\bar u)$. Clearly, this is equivalent to $T^{-1}$ being graphically regular at $\bar u$ for $\bar x$. In fact, for single-valued mappings, graphical regularity along with strict continuity is equivalent to strict differentiability as shown in the following result that we adapt from \cite[Exercise 9.25]{RoWe98}:
\begin{fact}\label{thm:clarke_reg_strict}
    Consider $T:D \to \bR^m$, $D \subset \bR^n$ and $\bar x \in \intr D$. $T$ is strictly differentiable at $\bar x$ if and only if $T$ is strictly continuous at $\bar x$ and graphically regular at $\bar x$.
\end{fact}
We now move on to the notion of proto-differentiability and its strict counterpart that will eventually lead to the strict differentiability of the generalized proximal mapping. Following \cite[Proposition 8.41]{RoWe98} we say that \emph{proto-differentiability} of $T$ at $\bar x$ for $\bar u$ is equivalent to the graph of $T$ being geometrically derivable at $(\bar x, \bar u)$, i.e. when the tangent and the derivative cones coincide, while \emph{strict proto-differentiability} is equivalent to $\gph T$ being strictly smooth at $(\bar x, \bar u)$ \cite{hang2025smoothness}, i.e. $\widehat{T}_{\gph T}(\bar x, \bar u) = \widetilde{T}_{\gph T}(\bar x, \bar u)$. We remark that in \cite{rockafellar1989proto} strict proto-differentiability was originally defined as equality of the \emph{tangent} cone (instead of the \emph{paratingent cone}) with the Clarke tangent cone, but we use the more widely accepted definition provided above. Since we will utilize this notion for mappings that are locally Lipschitzian, the weaker requirement presented in \cref{thm:clarke_reg_strict} will be enough. Further discussion on these notions can be found in \cite[Lemma 2.9]{gfrerer2025strict}.

Next we show that indeed when $g$ is strictly twice epi-differentiable, $\sigma$ from \cref{thm:lipschitz_sing_val_loc} is strictly differentiable, implying the hypothesis of \cref{thm:strict_fo_approx}.

\begin{lemma} \label{thm:sigma_strict_diff}
    Under the assumptions of \cref{thm:lipschitz_sing_val_loc}, if $g$ is strictly twice epi-differentiable at $\bar x$ for $\bar v$, then $\sigma$ from the same lemma is strictly differentiable at $0$.
\end{lemma}
\begin{proof}
    We will adapt the arguments of the proof of \cite[Theorem 4.2]{hang2025smoothness}. First, since $g$ is prox-regular at $\bar x$ for $\bar v$ and strictly twice epi-differentiable at $\bar x$ for $\bar v$, from \cite[Theorem 2.3]{poliquin1996generalized} the localization $T_\varepsilon$ is strictly proto-differentiable at $\bar x$ for $\bar v$.  Thus in light of \cite[Proposition 5.3]{hang2024role}, $\widetilde Q$ is strictly proto-differentiable at $\bar x$ for $\bar v - \nabla_x \Phi(\bar x, \bar y) = 0$. Moreover, strict proto-differentiability is a graphical property and as such $\sigma = \widetilde Q^{-1}$ is also strictly proto-differentiable at $0$ for $\bar x$ \cite[Proposition 5.1]{rockafellar1989proto}. 
    
    Now, note that strict proto-differentiability at $0$ for $\bar x$ implies $\widehat{T}_{\gph \sigma}(0, \bar x) = \widetilde T_{\gph \sigma}(0, \bar x)$. In light of the inclusion relations of the tangent cones \eqref{eq:cone_inclusion}, this further implies that $\widehat{T}_{\gph \sigma}(0, \bar x) = T_{\gph \sigma}(0, \bar x)$, i.e., that $\gph \sigma$ is Clarke regular at $(0, \bar x)$ which is another way of saying that $\sigma$ is graphically regular at $0$ \cite[Definition 8.38]{RoWe98}. Therefore, we have that $\sigma$ with $0 \in \intr \dom \sigma$ is Lipschitz continuous and graphically regular at $0$ and as such we can use \cref{thm:clarke_reg_strict} in order to get that $\sigma$ is strictly differentiable at $0$.
\end{proof}
We are now ready to show our key result applying the implicit function theorem of \cite{dontchev2009implicit} for generalized equations to obtain that the solution mapping $S$ has a Lipschitz continuous single-valued localization $s$ around $\bar y$ for $\bar x$. This localization is shown to coincide with the (generalized) proximal mapping $(\partial_\Phi g)^{-1}$.
If moreover, $g$ is strictly twice epi-differentiable, this localization happens to be strictly differentiable at $\bar y$.
\begin{proof}[Proof of \cref{thm:single_valued_prox}]
    ``\labelcref{thm:single_valued_prox:lipschitz}'': In light of \cref{thm:lipschitz_sing_val_loc}, $S$ has a Lipschitz continuous single-valued localization $s$ around $\bar y$ for $\bar x$. The next step is to show that $\argmin_x g(x) - \Phi(x, \bar y) \subseteq s(y)$ and that $s = \{x \in \bR^n: 0 \in T_\varepsilon(x) - \nabla_x \Phi(x, \bar y)\}$, where $T_\varepsilon$ is an $g$-attentive localization of $\partial g$ at $\bar x$ for $\bar v$.
    In light of \cref{thm:conj_cont} and since $\bar x \in (\partial_\Phi g)^{-1}(\bar y)$ is unique, we have that for any sequence $y^\nu \to \bar y$, there is $x^\nu \in (\partial_\Phi g)^{-1}(y^\nu)$ such that $x^\nu \to \bar x$ and $g^\Phi(y^\nu) \to g^\Phi(\bar y)$.

    By the modern version of Fermat's rule for the inner minimization problem in $(\partial_\Phi g)^{-1}(y^\nu)$, we have that $\nabla_x \Phi(x^\nu,y^\nu) \in \partial g(x^\nu)$ and since $\nabla_x \Phi$ is continuous on $X\times Y$, we have that $\|\nabla_x \Phi(x^\nu,y^\nu) - \nabla_x \Phi(\bar x,\bar y)\| \leq \varepsilon$, i.e. $\|v^\nu - \bar v\| \leq \varepsilon$ for $v^\nu \in \partial g(x^\nu)$, where $\varepsilon$ as in the definition of prox-regularity of $g$. Moreover, as $g^\Phi(y^\nu) = \Phi(x^\nu, y^\nu)-g(x^\nu) \to g^\Phi(\bar y) = \Phi(\bar x, \bar y)-g(\bar x)$ we have that $g(x^\nu) \to g(\bar x)$. Therefore, we can use the $g$-attentive localization and write 
    \begin{equation*}
        0 \in T_\varepsilon(x^\nu) - \nabla_x \Phi(x^\nu,y^\nu),
    \end{equation*}
    This implies that
    \begin{equation*}
        0 \neq \argmin_x g(x) - \Phi(x,y) \subseteq \{x \in \bR^n \mid 0 \in T_\varepsilon(x) - \nabla_x \Phi(x,y)\},
    \end{equation*}
    where we have replaced the subdifferential $\partial g$ with its $g$-attentive $\varepsilon$-localization $T_\varepsilon$. Now, for $y$ sufficiently near $\bar y$, since $s$ is a localization of $S$ around $\bar y$ for $\bar x$ we have that $\{x \in \bR^n: 0 \in T_\varepsilon(x) - \nabla_x \Phi(x,y)\} \subseteq s(y)$. Furthermore, as $s$ is single-valued and $(\partial_\Phi g)^{-1}(y) \neq \emptyset$ we get that $(\partial_\Phi g)^{-1}(y) = s(y)$ for a sufficiently small neighborhood of $\bar y$.
    
    ``\labelcref{thm:single_valued_prox:diff}'': If, moreover, $g$ is strictly twice epi-differentiable at $\bar x$ for $\bar v$, in light of \cref{thm:sigma_strict_diff}, the localization $\sigma$ of $Q^{-1}$ from \cref{thm:lipschitz_sing_val_loc} is also strictly differentiable at $0$. From \cref{thm:strict_fo_approx} we have that $\eta(y) = \sigma(\nabla_{yx}^2 \Phi(\bar x, \bar y)(y-\bar y))$ is a strict first-order approximation to $s$ at $\bar y$. Now, take $y, y'$ sufficiently close to $\bar y$:
    \begin{align*}
        &\|s(y) - s(y') - \nabla \sigma(0) \nabla_{yx}^2\Phi(\bar x, \bar y)(y-y')\| 
        \\
        & \qquad\leq \|s(y) - \eta(y) + \eta(y') - \sigma(y')\|
        \\
        &\qquad\quad+ \|\sigma(\nabla_{yx}^2 \Phi(\bar x, \bar y)(y-\bar y)) - \sigma(\nabla_{xy}^2 \Phi(\bar x, \bar y)(y'-\bar y)) - \nabla \sigma(0) \nabla_{yx}^2\Phi(\bar x, \bar y)(y-y')\|
        \\
        & \qquad\leq \varepsilon \|y-y'\| + \varepsilon \|\nabla_{yx}^2\Phi(\bar x,\bar y)\|\|y-y'\|
    \end{align*}
    for any $\varepsilon > 0$. The first inequality follows by the triangle inequality while the second by the fact that $\eta$ is a strict first-order approximation to $s$ and the strict differentiability of $\sigma$. Therefore, since $\varepsilon$ is arbitrary, $s$ is strictly differentiable at $\bar y$ with $\nabla s(\bar y) = \nabla \sigma(0) \nabla_{xy}^2\Phi(\bar x,\bar y)$. But in \cref{thm:single_valued_prox:lipschitz} we identified $s$ with $(\partial_\Phi g)^{-1}$ locally and as such the same holds for $(\partial_\Phi g)^{-1}$.
    
    ``\labelcref{thm:single_valued_prox:diff_prox_bound}'': This result follows directly by \labelcref{thm:single_valued_prox:diff} along with the extra result of \cref{thm:lipschitz_sing_val_loc} under prox-boundedness of $g$.
\end{proof}

Having obtained conditions that guarantee the continuity and differentiability of the generalized proximal mapping, we now move on to the corresponding properties of the $\Phi$-conjugate:
\begin{proposition}[Lipschitz differentiability and strict twice differentiability of $\Phi$-conjugates]
\label{thm:env_diff}
Let the assumptions of \cref{thm:single_valued_prox} hold true. Then,
    \begin{propenum}
        \item \label{thm:env_diff:lipschitzdiff}  $g^\Phi$ is locally Lipschitz differentiable for $y$ sufficiently near $\bar y$;
        \item \label{thm:env_diff:2nddiff} if $g$ is strictly twice epi-differentiable at $\bar x$ for $\bar v \in \partial g(\bar x)$ and prox-bounded with threshold $\lambda_g > \lambda_{\max}(M)$, $g^\Phi$ is strictly twice differentiable at $\bar y$ and its Hessian is given by
        \begin{equation} \label{eq:hessian_env}
            \nabla^2 g^\Phi(\bar y) = \nabla (\partial_\Phi g)^{-1}(\bar y)\nabla_{yx}^2 \Phi((\partial_\Phi g)^{-1}(\bar y),\bar y) + \nabla_{yy}^2 \Phi((\partial_\Phi g)^{-1}(\bar y),\bar y).
        \end{equation}
    \end{propenum}
\end{proposition}
\begin{proof}
``\labelcref{thm:env_diff:lipschitzdiff}'':
In light of \cref{thm:equivalence_phi_prox_reg}, the gradient of the $\Phi$-conjugate has the form given in \eqref{eq:env_grad}. Since $\Phi$ is twice-differentiable, $\nabla_y \Phi$ is locally Lipschitz and as $(\partial_\Phi g)^{-1}$ is Lipschitz close to $\bar y$, $\nabla g^\Phi$ is also Lipschitz in a sufficiently small neighborhood of $\bar y$.

``\labelcref{thm:env_diff:2nddiff}'': The strict twice differentiability of the envelope follows from \labelcref{thm:env_diff:lipschitzdiff} along with the strict differentiability of $(\partial_\Phi g)^{-1}$ at $\bar y$ invoking \cref{thm:single_valued_prox:diff_prox_bound} and the assumptions on $\Phi$. The form of the Hessian matrix is given by differentiating \eqref{eq:env_grad}.
\end{proof}
\begin{remark}[Eigenvalue upper bound as a parameter]
    It is crucial to note that in many important cases $\lambda_{\max}(M)$ in \cref{thm:single_valued_prox} is a parameter that can be controlled. Consider the setting of \cref{example:quadratic} and notice that $\lambda_{\max}(M) = \gamma$ for $\Phi(x,y) = -\tfrac{1}{2\gamma}\|x-y\|^2$. By taking $\gamma$ small enough we can guarantee that $\lambda_{\max}(M) < r^{-1}$. It is straightforward that we have similar control over the eigenvalue upper bound in other important settings as well. For example in the anisotropic case where $\Phi(x,y) = -\gamma \star \phi(x-y)$ we get $\nabla^2_{xx} \Phi(x,y) = -\tfrac{1}{\gamma}\nabla^2 \phi(\tfrac{1}{\gamma}(x-y))$ and thus for a strongly convex $\phi$ we have $\nabla^2 \phi(\tfrac{1}{\gamma}(x-y)) \geq \sigma > 0$ and we can tune the parameter by choosing a suitable $\gamma$. 
\end{remark}

\Cref{thm:single_valued_prox,thm:env_diff} generalize important results on the differentiability of popular envelope functions and lead to new results for less explored ones. In the setting of \cref{example:quadratic}, where $\Phi(x,y) = -\tfrac{1}{2\gamma}\|x-y\|^2$, we retrieve the classical strict differentiability of the proximal mapping and the strict twice differentiability of the Moreau envelope \cite[Theorem 4.1]{poliquin1996generalized}. Substituting $\Phi$ in \eqref{eq:hessian_env} we obtain $\nabla^2 g^\Phi(\bar y) = \tfrac{1}{\gamma}(\nabla (\partial_\Phi g)^{-1}(\bar y)-I)$, matching the relation of the Hessian of the envelope and the Jacobian of the prox from \cite[Theorem 3.8]{poliquin1996generalized}. In the setting of \cref{example:left_bregman}, where $\Phi(x,y) = -\tfrac{1}{\gamma}D_h(x,y)$, our results generalize \cite[Theorems 3.10 and 3.11]{ahookhosh2021bregman} on the Lipschitz continuity and strict differentiability of the left Bregman proximal mapping and the strict twice differentiability of the left Bregman--Moreau envelope. Again, substituting $\Phi$ in \eqref{eq:hessian_env} we get $\nabla^2 g^\Phi(\bar y) = \tfrac{1}{\gamma}\nabla^2 h(\bar y)(\nabla (\partial_\Phi g)^{-1}(\bar y)-I)$, matching \cite[Theorem 3.11(iii)]{ahookhosh2021bregman}. Regarding the anisotropic and the $\varphi$-divergence couplings, \cref{example:aniso,example:phi_div}, to the best of our knowledge our results are the first describing conditions under which the generalized envelope is strictly twice differentiable.


\section{Conclusion}
In this paper we studied the regularity properties of generalized conjugates, providing verifiable conditions for continuous differentiability. Moreover, we established conditions for strict twice differentiability, utilizing and extending a nonsmooth version of the implicit function theorem. There are many open avenues worth devoting future investigations. A generalized forward-backward envelope linesearch procedure akin to \cite{themelis2018forward} for additive composite problems where the smooth term satisfies a very general notion of smoothness is currently pursued. Another interesting research direction is the utilization of the smoothness properties of generalized conjugates to obtain convergence guarantees for subgradient methods under $\Phi$-convexity in the manner of \cite{davis2019stochastic}.



\printbibliography
\end{document}